\documentclass[12pt]{amsart}

\usepackage{amscd,amssymb,stmaryrd,amsmath,graphicx,verbatim,xypic}
\usepackage{wasysym}
\usepackage[TS1,OT1,T1]{fontenc}
\usepackage[applemac]{inputenc}
\usepackage{lscape} 
\usepackage{fullpage} 
\usepackage{pslatex} 
\usepackage{tikz}
\usepackage{hyperref}
\usetikzlibrary{shapes,arrows}

\newtheorem{theorem}{Theorem}[section]
\newtheorem{lemma}[theorem]{Lemma}
\newtheorem{corollary}[theorem]{Corollary}
\newtheorem{proposition}[theorem]{Proposition}

\theoremstyle{definition}
\newtheorem{definition}[theorem]{Definition}

\newtheorem{example}[theorem]{Example}

\theoremstyle{remark}
\newtheorem{remark}[theorem]{Remark}

 \DeclareMathOperator{\Hom}{Hom}

\newcommand{\lten}{\overset{{\mathbb{L}}}{\otimes}}
\newcommand{\RHom}{\mathrm{\mathbb{R}Hom}}

\newcommand{\Dcal}{\ensuremath{\mathcal{D}}}
\newcommand{\Xcal}{\ensuremath{\mathcal{X}}}
\newcommand{\Ycal}{\ensuremath{\mathcal{Y}}}
\newcommand{\Tcal}{\ensuremath{\mathcal{T}}}
\newcommand{\Fcal}{\ensuremath{\mathcal{F}}}
\newcommand{\Rcal}{\ensuremath{\mathcal{R}}}
\newcommand{\Ccal}{\ensuremath{\mathcal{C}}}

\newcommand{\Acal}{\ensuremath{\mathcal{A}}}

\newcommand{\Kcal}{\ensuremath{\mathcal{K}}}

\newcommand{\Kbb}{\ensuremath{\mathbb{K}}}

\newcommand{\Zbb}{\ensuremath{\mathbb{Z}}}

\newcommand{\wt}{\widetilde}

\newcommand{\thick}{\mathrm{thick}}

\newcommand{\K}{\mathbb{K}}
\newcommand{\Z}{\mathbb{Z}}

\newcommand{\ra}{\rightarrow}

\numberwithin{equation}{section}

\begin{document}
\bibliographystyle{amsplain}
\title{Glueing silting objects}
\author{Qunhua Liu, Jorge Vit{\'o}ria, Dong Yang}\thanks{The second respectively third named authors were supported by DFG - SPP Darstellungstheorie 1388 KO1281/7-1 respectively KO1281/9-1.}
\thanks{The authors would like to thank David Pauksztello for helpful remarks and thank David Pauksztello and Lutz Hille for a discussion which led to example \ref{example complete to tilting}.}
\address{Qunhua Liu, Institute of Mathematics, School of Mathematical Sciences, Nanjing Normal University, Nanjing 210023, P.R.China} \email{05402@njnu.edu.cn}
\address{Dong Yang, Department of Mathematics, Nanjing University, 22 Hankou Road, Nanjing 210093, P. R. China}
\email{dongyang2002@gmail.com}
\address{Jorge Vit{\'o}ria, Department of computer science - sector of mathematics, Universit\`a degli Studi di Verona, Strada le Grazie 15 - Ca' Vignal, I-37134 Verona, Italy} 
\email{jorge.vitoria@univr.it}


\begin{abstract}
Recent results by Keller and Nicol{\'a}s and by Koenig and Yang have shown bijective correspondences between suitable classes of t-structures and co-t-structures with certain objects of the derived category: silting objects. On the other hand, the techniques of glueing (co-)t-structures along a recollement play an important role in the understanding of derived module categories. Using the above correspondence with silting objects we present explicit constructions of glueing of silting objects, and, furthermore, we answer the question of when the glued silting is tilting.
\end{abstract}
\maketitle
\section{Introduction}

In a given triangulated category, the study of its torsion pairs helps to understand its structure. Two kinds of torsion pairs have been considered with particular emphasis in the literature. These are the notions of t-structure (introduced by Beilinson, Bernstein and Deligne in \cite{BeilinsonBernsteinDeligne82}) and co-t-structure (introduced independently by Bondarko in \cite{Bondarko10} and Pauksztello in \cite{Pauksztello08}). These are torsion pairs with an additional property concerning the suspension functor of the underlying triangulated category and they give rise to additive (or even abelian in the case of t-structures) subcategories which are of interest. In this paper, we work with correspondences that classify these torsion pairs in terms of objects of the triangulated category.

Keller and Vossieck established in \cite{KellerVossieck88} a bijection between bounded t-structures and equivalence classes of silting objects in the bounded derived category of modules over the path algebra of a Dynkin quiver over a field. Recently, this bijection has been extended by Keller and Nicol\'as in \cite{KellerNicolas11} for the bounded derived categories of homologically homologically smooth non-positive differential graded algebras and by Koenig and Yang in \cite{KoenigYang12} for bounded derived categories of finite dimensional algebras over a field. 
Indeed, they show that in such a category, there is a bijection between silting objects and bounded t-structures whose hearts are length categories.
It turns out that Keller and Vossieck's result is a corollary, since for algebras of finite representation type all hearts of bounded t-structures are length categories. 
A new correspondence between silting objects and bounded co-t-structures was proved in \cite{MendozaSaenzSantiagoSouto10,KellerNicolas11}. This bijection will be central in our approach.

Silting objects play, thus, a more general role than tilting objects. They describe all hearts which are length categories and these turn out to be precisely those which are module categories over some finite dimensional algebra over a field. This algebra, although not in general derived equivalent to the one we started with, is obtained as the endomorphism algebra of the silting object, just like in the tilting setting. Indeed, it is easy to observe that a silting object is tilting if and only if it lies in the heart of the corresponding t-structure. 

Glueing techniques with respect to a recollement, due to Beilinson, Bernstein and Deligne, have been intensively studied in \cite{BeilinsonBernsteinDeligne82} for t-structures and, recently, in \cite{Bondarko10} for co-t-structures.  This leads to the natural question of how to glue silting objects and which silting objects are glued from \textit{smaller} ones. Indeed, recent work by the two first authors has shown that, in the piecewise hereditary case, all bounded t-structures whose heart is a length category are glued with respect to a non-trivial recollement \cite{LiuVitoria12}. In this setting it is then clear that every silting object can be decomposed by this process into as many pieces as derived simple factors of the algebra (check \cite{AngeleriKoenigLiu12a} and \cite{AngeleriKoenigLiu12} for terminology). It turns out, however, that an answer to the problem of glueing silting objects can be given more easily when the focus is on co-t-structures rather than on t-structures. Our main result is as follows.
\\

\noindent\textbf{Theorem} (Theorem \ref{glue silting})
\textit{Let $\Rcal$ be a recollement of a triangulated category $\Dcal$ of the form 
\begin{equation*}
\begin{xymatrix}{\mathcal{Y}\ar[r]^{i_*}&\mathcal{D}\ar@<3ex>[l]_{i^!}\ar@<-3ex>[l]_{i^*}\ar[r]^{j^*}&\mathcal{\mathcal{X}}\ar@<3ex>_{j_*}[l]\ar@<-3ex>_{j_!}[l]}.
\end{xymatrix}
\end{equation*}
 Let $X$ and $Y$ be respectively silting objects of $\Xcal$ and $\Ycal$ and $(\Xcal_{\geq 0},\Xcal_{\leq 0})$ and $(\Ycal_{\geq 0},\Ycal_{\leq 0})$ be respectively the associated co-$t$-structure in $\Xcal$ and $\Ycal$. Then the induced co-t-structure $(\Dcal_{\geq 0},\Dcal_{\leq 0})$ in $\Dcal$ is associated with the silting object $Z=i_*Y\oplus K_X$ in $\Dcal$, with $K_X$ defined by the triangle
\begin{equation}\nonumber
i_*\beta_{\geq 1}i^!j_!X\longrightarrow j_!X\longrightarrow K_X\longrightarrow (i_*\beta_{\geq 1}i^!j_!X)[1],
\end{equation}
where $\beta_{\geq 1}$ is a (non-functorial) choice of truncation for the co-t-structure $(\Ycal_{\geq 0},\Ycal_{\leq 0})$ in $\Ycal$.
}
\\

Furthermore, the question of when we can  glue derived equivalences, i.e., tilting objects, comes as a particular setting of the general context of glueing silting. Similar constructions of tilting objects have been discussed in \cite{Ladkani11} and \cite{AngeleriKoenigLiu11a}. In particular, we will show that the construction in \cite{AngeleriKoenigLiu11a} is a particular case of the construction above. The following is our main theorem concerning tilting. 
\\

\textbf{Theorem} (Theorem \ref{conditions tilting})
\textit{Let $\Rcal$ be a recollement of $\Dcal=\Dcal^b(R)$ by  $\Xcal=\Dcal^b(C)$ and $\Ycal=\Dcal^b(B)$, where $R$, $C$ and $B$ are finite-dimensional algebras over a field of finite global dimension. Let $X$ and $Y$ be tilting objects of $\Xcal$ and $\Ycal$, respectively. Then $Z=i_*Y\oplus K_X$ is tilting in $\Dcal$ if and only if the following conditions hold.
\begin{enumerate}
\item[(a)] $\Hom_{\Ycal}(Y,i^*j_*X[k])=0$ for all $k<-1$;
\item[(b)] $\Hom_{\Ycal}(i^*j_*X,Y[k])=0$ for all $k<0$;
\item[(c)] $\Hom_{\Ycal}(i^*j_*X,i^*j_*X[k])=0$ for all $k<-1$.\\
\end{enumerate}}

This paper is structured as follows. In the next section we discuss some preliminary results on recollements, t-structures, co-t-structures and silting needed for the later sections. In section 3 we show how to glue silting and we use this in section 4 to give necessary and sufficient conditions for the glued silting to be tilting. These conditions are particularly nice in the hereditary case and they are made explicit in section 5. Finally, in section 6, we make a few observations on the glueing of HRS-tilts with view towards a comment on the compatibility of silting mutation with glueing.

\section{Preliminaries}
Throughout, $\mathbb{K}$ denotes a fixed field and $R$, $B$ and $C$ finite dimensional $\Kbb$-algebras.  The bounded derived category of finitely generated right $R$-modules and the bounded homotopy category of finitely generated projective right $R$-modules  will be denoted by, respectively, $\Dcal^b(R)$ and $\Kcal^b(proj\mbox{-}R)$. The symbols $\Xcal$, $\Ycal$ and $\Dcal$ denote triangulated categories.
\subsection{Recollements}
A \textbf{recollement} of $\Dcal$ by $\Xcal$ and $\Ycal$ is a diagram of six triangle
functors
\begin{equation}\label{recollement}
\begin{xymatrix}{\mathcal{Y}\ar[r]^{i_*}&\mathcal{D}\ar@<3ex>[l]_{i^!}\ar@<-3ex>[l]_{i^*}\ar[r]^{j^*}&\mathcal{\mathcal{X}}\ar@<3ex>_{j_*}[l]\ar@<-3ex>_{j_!}[l]}.
\end{xymatrix}
\end{equation}
satisfying the following properties:
\begin{enumerate}
\item $(i^\ast,i_\ast)$,\,$(i_*,i^!)$,\,$(j_!,j^*)$ ,\,$(j^\ast,j_\ast)$
are adjoint pairs;

\item $i_\ast,\,j_\ast,\,j_!$  are full embeddings;

\item  $i^!\circ j_\ast=0$ (and thus also $j^*\circ i_*=0$ and
$i^\ast\circ j_!=0$);

\item for each $Z\in \Dcal$ the units and counits of the adjunctions yield triangles \[i_* i^!Z\to
Z\to j_\ast j^\ast Z\to i_* i^!Z[1]\]
\[j_! j^* Z\to Z\to
i_\ast i^\ast Z\to j_!j^*Z[1].\]
\end{enumerate}

The following result allows us to change the sides of a recollement.
\begin{theorem} \emph{(\cite[propositions 3.6 and 3.7]{BondalKapranov89}, \cite[proposition 5 and theorem 7]{Jorgensen10})}
\label{reflect}
Let $\Rcal$ be a recollement of $\Dcal$ of the form (\ref{recollement}). If $\Dcal$ has a Serre functor $S$, then both $\Xcal$ and $\Ycal$ have Serre functors ($S_\Xcal$ and $S_\Ycal$ respectively) and there are reflected recollements $\Rcal_U$ and $\Rcal_L$ (the upper and the lower reflection, respectively)
\begin{equation}\nonumber
\Rcal_U: \begin{xymatrix}{\mathcal{X}\ar[r]^{j_!}&\mathcal{D}\ar@<3ex>[l]_{j^*}\ar@<-3ex>[l]_{j^\#}\ar[r]^{i^*}&\mathcal{\mathcal{Y}}\ar@<3ex>_{i_*}[l]\ar@<-3ex>_{i_\#}[l]},
\end{xymatrix}
\end{equation}
where $i_\#=S^{-1}i_*S_\Ycal$ and $j^\#=S_\Xcal^{-1}j^*S$ are left adjoints of $i^*$ and $j_!$ respectively, and
\begin{equation}\nonumber
\Rcal_L: \begin{xymatrix}{\mathcal{X}\ar[r]^{j_*}&\mathcal{D}\ar@<3ex>[l]_{j^+}\ar@<-3ex>[l]_{j^*}\ar[r]^{i^!}&\mathcal{\mathcal{Y}}\ar@<3ex>_{i_+}[l]\ar@<-3ex>_{i_*}[l]}.
\end{xymatrix}
\end{equation}
where $i_+=Si_*S_\Ycal^{-1}$ and $j^+=S_\Xcal j^*S^{-1}$ are right adjoints of $i^!$ and $j_*$ respectively.
\end{theorem}

\subsection{t-structures and co-t-structures}
A \textbf{torsion pair} in $\Dcal$ is a pair of strict full subcategories $(\Dcal',\Dcal'')$, closed under taking direct summands and direct sums, such that
\begin{enumerate}
\item $\Hom_{\Dcal}(X,Y)=0$ for all $X\in \Dcal'$ and $Y\in\Dcal''$;
\item for each $Z$ in $\Dcal$, there are $X\in\Dcal'$, $Y\in\Dcal''$ and a triangle
\begin{equation}\label{tria torsion pair}
X \ra Z \ra Y \ra X[1].
\end{equation}
\end{enumerate}
By definition, it is easy to see that
\begin{equation}\nonumber
\Dcal' = \{X \in \Dcal: \Hom_{\Dcal}(X,Y) =
0,\ \forall\ Y\in \Dcal'' \}=:\ ^\perp\Dcal'',
\end{equation} 
\begin{equation}\nonumber
\Dcal'' = \{Y \in
\Dcal: \Hom_{\Dcal}(X,Y) = 0,\ \forall\ X\in \Dcal' \}=:\ \Dcal'^{\perp}.~~\ 
\end{equation}

A {\bf t-structure} in $\Dcal$ is a pair $(\Dcal^{\leq 0}, \Dcal^{\geq 0})$  such that $(\Dcal^{\leq 0}, \Dcal^{\geq 1})$ is a torsion pair and $\Dcal^{\leq 0}\subseteq \Dcal^{\leq 1}$, where $\Dcal^{\leq n}:=\Dcal^{\leq 0}[-n]$ and $\Dcal^{\geq n}:=\Dcal^{\geq 0}[-n]$ for $n\in\mathbb{Z}$.
The
subcategory $\Dcal^{\leq 0}$ is called the {\em t-aisle}, and
$\Dcal^{\geq 0}$ is called the \textbf{t-coaisle}. 
  The \textbf{heart}
$\Dcal^{\leq 0} \cap \Dcal^{\geq 0}$ is always an abelian category. We say the heart is $\textbf{length}$ if it has only finitely many simple objects (up to isomorphism) and each object in the heart has finite length. A t-structure is said to be \textbf{bounded} if it satisfies 
$$ \bigcup\limits_{n\in\Zbb} D^{\geq n}= \Dcal = \bigcup\limits_{n\in\Zbb} D^{\leq n}.$$
For all $n\in\Zbb$, there is a right adjoint to the inclusion of the subcategory $\Dcal^{\leq n}$ in $\Dcal$, called \textbf{the truncation at $n$} and denoted by $\tau^{\leq n}$. 
Similarly, there is a left adjoint to the inclusion of $\Dcal^{\geq n}$, denoted by $\tau^{\geq n}$. In fact, the triangle (\ref{tria torsion pair}) can be written functorially, for any $Z\in\Dcal$, as follows
\begin{equation}\nonumber
\tau^{\leq 0}Z \ra Z \ra \tau^{\geq 1}Z \ra (\tau^{\leq 0}Z)[1].
\end{equation}
Moreover, it is possible to define associated cohomological functors 
$$H^i: \Dcal\rightarrow \Dcal^{\leq 0}\cap\Dcal^{\geq 0},\ \ \  H^i(X)=(\tau^{\leq i}\tau^{\geq i}X)[i], \forall X\in\Dcal.$$
If a t-structure is bounded then, for any object $X\in\Dcal$, $H^i(X)=0$ for all but finitely many $i\in\Zbb$. A well-known example is the standard t-structure in $\Dcal^b(R)$.
 The associated cohomological functors are given by the usual cohomology of complexes in $\Dcal^b(R)$.

Similarly, a \textbf{co-t-structure} in $\Dcal$ is a pair  $(\Dcal_{\geq 0}, \Dcal_{\leq 0})$ such that $(\Dcal_{\geq 0},\Dcal_{\leq -1})$ is a torsion pair and $\Dcal_{\geq 0}\subseteq\Dcal_{\geq -1}$, where 
$\Dcal_{\geq n} := \Dcal_{\geq 0}[-n]$ and $\Dcal_{\leq n} :=
\Dcal_{\leq 0}[-n]$, for all $n\in \Z$.
The subcategory $\Dcal_{\geq 0}$ is called the \textbf{co-t-aisle}, and
$\Dcal_{\leq 0}$ is called the \textbf{co-t-coaisle}. The {\bf co-heart}
$\Dcal_{\geq 0} \cap \Dcal_{\leq 0}$ is not, in general, an abelian category, as in the case of t-structures, but it still has the structure of an additive category. A co-t-structure is said to be \textbf{bounded} if it satisfies 
$$ \bigcup\limits_{n\in\Zbb} D_{\geq n}= \Dcal = \bigcup\limits_{n\in\Zbb} D_{\leq n}.$$
Contrarily to t-structures, the triangle obtained from the torsion pair $(\Dcal_{\geq 0},\Dcal_{\leq -1})$ is not functorial. However, a choice of $X\in \Dcal_{\geq 0}$ (respectively, of $Y\in\Dcal_{\leq -1}$) as in triangle (\ref{tria torsion pair}) will be denoted by $\beta_{\geq 0}Z$ (respectively, by $\beta_{\leq -1}Z$). A well-known example (see \cite{Bondarko10}, \cite{Pauksztello08}) is the standard co-t-structure in $\Kcal^b(proj\mbox{-}R)$. 
Here, the non-functorial choices of $\beta_{\geq 0}Z$ or $\beta_{\leq 0}Z$ for some object $Z\in \Kcal^b(proj\mbox{-}R)$ are usually given by the \textbf{stupid truncations}, consisting in setting suitable entries of the complex $Z$ to be zero.

It is known that t-structures and co-t-structures can be glued (or induced) with respect to a recollement (see \cite{BeilinsonBernsteinDeligne82} for t-structures and \cite{Bondarko10} for co-t-structures). The same arguments can be used to prove the following result.
\begin{theorem}\label{glue}
Let  $\Rcal$ be a recollement of $\Dcal$ of the form (\ref{recollement}). Let $(\Xcal',\Xcal'')$ and $(\Ycal',\Ycal'')$ be torsion pairs in $\Xcal$ and $\Ycal$, respectively. 
\begin{enumerate} \item There is a glued torsion pair $(\Dcal',\Dcal'')$ in $\Dcal$ defined by:
\begin{equation}\nonumber
\Dcal'=\left\{Z\in \Dcal: j^*Z\in \Xcal', i^*Z\in\Ycal'\right\},\ \ \ \Dcal''=\left\{Z\in \Dcal: j^*Z\in \Xcal'', i^!Z\in\Ycal''\right\}.
\end{equation}
\item\emph{(\cite[theorem 1.4.10]{BeilinsonBernsteinDeligne82})} If $(\Xcal',\Xcal''[1])$ and $(\Ycal',\Ycal''[1])$ are t-structures, then $(\Dcal',\Dcal''[-1])$  is a t-structure in $\Dcal$;
\item\emph{(\cite[theorem 8.2.3]{Bondarko10})}  If $(\Xcal',\Xcal''[-1])$ and $(\Ycal',\Ycal''[-1])$ are co-t-structures, then $(\Dcal',\Dcal''[1])$ is a co-t-structure in $\Dcal$.
\item\emph{(\cite{Bondarko10}, \cite{LiuVitoria12}, \cite{Wang06})} The glueing of bounded t-structures whose heart is a length category is a bounded t-structure whose heart is a length category. Also, the glueing of bounded co-t-structures is still bounded.
\end{enumerate}
\end{theorem} 

If $\Dcal$ admits a Serre functor $S$, $\Rcal$ is a recollement of $\Dcal$ of the form (\ref{recollement}) and $(\Xcal^{\leq 0},\Xcal^{\geq 0})$ and $(\Ycal^{\leq 0},\Ycal^{\geq 0})$ are t-structures in $\Xcal$ and $\Ycal$ respectively, then there are three naturally associated t-structures in $\Dcal$, obtained by glueing these with respect to $\Rcal$, $\Rcal_U$ and $\Rcal_L$ - we will denote them by $(\Dcal^{\leq 0},\Dcal^{\geq 0})$, $(\Dcal_U^{\leq 0},\Dcal_U^{\geq 0})$, $(\Dcal_L^{\leq 0},\Dcal_L^{\geq 0})$, respectively. Similarly, given co-t-structures $(\Xcal_{\geq 0},\Xcal_{\leq 0})$ and $(\Ycal_{\geq 0},\Ycal_{\leq 0})$ in $\Xcal$ and $\Ycal$, we get three glued co-t-structures $(\Dcal_{\geq 0},\Dcal_{\leq 0})$, $(\Dcal^U_{\geq 0},\Dcal^U_{\leq 0})$, $(\Dcal^L_{\geq 0},\Dcal^L_{\leq 0})$.

\subsection{Correspondences with silting objects}

An object $M\in\Dcal$ is \textbf{silting} if $\Hom_{\Dcal}(M, M[i])=0$, for all $i>0$ and $\Dcal$ is the smallest triangulated subcategory containing $M$ which is closed under direct summands, extensions and shifts. It is \textbf{tilting} if, in addition, $\Hom_{\Dcal}(M,M[i])=0$, for all $i<0$. We say that two such objects $M$ and $N$ are equivalent if $add(M)=add(N)$.   The following result is a consequence of~\cite[corollary 5.9]{MendozaSaenzSantiagoSouto10} (see also \cite{KellerNicolas11}).

\begin{theorem}\label{general bijection}
There is a bijection between bounded co-t-structures in $\Dcal$ whose co-heart is additively generated by one object and equivalence classes of silting objects of $\Dcal$.
\end{theorem}

This bijection can be described as follows.
\begin{itemize}
\item Given a co-t-structure $(\Dcal_{\geq 0},\Dcal_{\leq 0})$, the associated equivalence class is given by a silting object $M$ which is an additive generator of the co-heart, i.e., $\Dcal_{\geq 0}\cap\Dcal_{\leq 0}=add(M)$.
\item Given a silting object $M\in\Dcal$, the associated co-t-structure is defined as follows:
$\Dcal_{\geq 0}$ is the smallest full subcategory closed under direct summands, direct sums and extensions containing $\left\{M[i]:i\leq 0\right\}$; similarly, $\Dcal_{\leq 0}$ is the smallest full subcategory closed under direct summands, direct sums and extensions containing $\left\{M[i]:i\geq 0\right\}$. 
\end{itemize}

\begin{theorem}\emph{(\cite[theorem 7.1]{KoenigYang12}).}\label{bijections}
In $\Dcal^b(R)$, there are bijections between
\begin{itemize}
\item the set of bounded t-structures in $\Dcal^b(R)$ whose heart is a length category;
\item the set of equivalence classes of silting objects in $K^b(proj\mbox{-}R)$;
\item the set of bounded co-t-structures in $K^b(proj\mbox{-}R)$.
\end{itemize}
\end{theorem}

Recall that $R$ is of finite global dimension if and only if $\Dcal^b(R)$ admits a Serre functor (\cite{BondalKapranov89}, \cite{ReitenVandenBergh02}). Also, in this case, we have $K^b(proj\mbox{-}R)\cong \Dcal^b(R)$ and, thus, the correspondence between t-structures and co-t-structures occurs in the same category. 
Under this assumption, the bijections of the theorem can be made explicit as follows (see \cite[section 6]{KoenigYang12} for details).

\begin{itemize}
\item Given a t-structure $(\Dcal^{\leq 0},\Dcal^{\geq 0})$ in $\Dcal^b(R)$, we can associate a co-t-structure $(\Dcal_{\geq 0},\Dcal_{\leq 0})$ by taking $\Dcal_{\leq 0}=\Dcal^{\leq 0}$ and $\Dcal_{\geq 0}=\ ^\perp\Dcal^{\leq -1}$; In the language of \cite{Bondarko10}, this is the \textbf{left adjacent co-t-structure} to the t-structure $(\Dcal^{\leq 0},\Dcal^{\geq 0})$.
\item Given a co-t-structure $(\Dcal_{\geq 0},\Dcal_{\leq 0})$ in $\Dcal^b(R)$, we can associate a t-structure by taking $\Dcal^{\leq 0}=\Dcal_{\leq 0}$ and $\Dcal^{\geq 0}=\Dcal_{\leq -1}^\perp$; In the language of \cite{Bondarko10}, this is the \textbf{right adjacent t-structure} to the co-t-structure $(\Dcal_{\geq 0},\Dcal_{\leq 0})$.
\item Given a t-structure $(\Dcal^{\leq 0},\Dcal^{\geq 0})$ in $\Dcal^b(R)$, the associated silting object is the direct sum of the indecomposable Ext-projective objects in the aisle, i.e., the objects $X$ of $\Dcal^{\leq 0}$ such that $\Hom_{\Dcal^b(R)}(X,Y[1])=0$ for all $Y\in\Dcal^{\leq 0}$.
\item Given a silting object $M$ in $\Dcal^b(R)$, the associated t-structure is defined as follows
\begin{equation}\nonumber
\Dcal^{\leq 0}=\left\{Z\in\Dcal^b(R): \Hom_{\Dcal^b(R)}(M,Z[i])=0, \forall i>0\right\}
\end{equation}
\begin{equation}\nonumber
\Dcal^{\geq 0}=\left\{Z\in\Dcal^b(R): \Hom_{\Dcal^b(R)}(M,Z[i])=0, \forall i<0\right\}.
\end{equation}
\end{itemize}
\begin{remark} If $R$ is of finite global dimension, as before, then the co-t-coaisle $\Dcal_{\leq 0}$ associated with a silting object $M$ in $\Dcal^b(R)$ coincides with the aisle $\Dcal^{\leq 0}$ associated with the same silting object. In other words, the co-t-structure associated to $M$ is left adjacent to the t-structure associated to $M$.
\end{remark}

\begin{remark}\label{adjoint structures}
The structures glued from adjacent t-structures and co-t-structures are not directly related by adjacency, as observed by Bondarko~\cite[remark 8.2.4.4]{Bondarko10}. Suppose, however, that $\Dcal$ has a Serre functor and that $(\Xcal_{\geq 0},\Xcal_{\leq 0})$ (respectively, $(\Ycal_{\geq 0},\Ycal_{\leq 0})$) is a left adjacent co-t-structure to the t-structure $(\Xcal^{\leq 0},\Xcal^{\geq 0})$ (respectively, $(\Ycal^{\leq 0},\Ycal^{\geq 0})$). Using the descriptions provided in theorem \ref{glue} and the functors of theorem \ref{reflect} for a fixed recollement $\Rcal$, we have
\begin{equation}\nonumber
\Dcal_{\leq 0}^U=\Dcal^{\leq 0}=\left\{Z\in \Dcal: j^*Z\in\Xcal^{\leq 0}=\Xcal_{\leq 0}, i^*Z\in\Ycal^{\leq 0}=\Ycal_{\leq 0}\right\}.
\end{equation}
This means that the co-structure $(\Dcal^U_{\geq 0},\Dcal^U_{\leq 0})$ glued by the upper reflected recollement $\Rcal_U$ is left adjacent to the t-structure $(\Dcal^{\leq 0},\Dcal^{\geq 0})$ glued by the recollement $\Rcal$. Similarly, the t-structure $(\Dcal_L^{\leq 0},\Dcal_L^{\geq 0})$ glued by  $\Rcal_L$ is right adjacent to the co-t-structure $(\Dcal_{\geq 0},\Dcal_{\leq 0})$ glued by $\Rcal$.
\end{remark}

\begin{remark}\label{tilting}
A tilting object $T$ in $\Kcal^b(proj\mbox{-}R)$ yields equivalences between $\Dcal^b(R)$ and $\Dcal^b(End(T))$ and between $\Kcal^b(proj\mbox{-}R)$ and $\Kcal^b(proj\mbox{-}End(T))$ (see \cite[theorem 6.4]{Rickard89}). Under these equivalences, the t-structure and the co-t-structure associated to $T$ in $\Dcal^b(R)$ correspond to the standard ones in $\Dcal^b(End(T))$ and  $K^b(proj\mbox{-}End(T))$, respectively.
\end{remark}

\subsection{HRS-tilts and silting mutation}\label{HRS-tilts and silting mutation}

There are mutation operations on both t-structures and silting objects. Recall that the definition of a torsion pair in an abelian category is analogous to that of a torsion pair in a triangulated category, replacing the triangle axiom by a short exact sequence.

\begin{theorem}\emph{(\cite[proposition 2.1]{ HappelReitenSmaloe96}, \cite[proposition 2.5]{Bridgeland05}).}\label{HRS tilting}
Let $(\Dcal^{\leq 0},\Dcal^{\geq 0})$ be a bounded t-structure in a triangulated category $\Dcal$ with heart $\mathcal{A}$ and associated cohomology functors $H^i$, $i\in\Zbb$. Suppose that  $(\mathcal{T},\mathcal{F})$ is a torsion pair in $\mathcal{A}$. Then $(\mathcal{D}_{(\Tcal,\Fcal)}^{\leq 0}, \mathcal{D}_{(\Tcal,\Fcal)}^{\geq 0})$ is a t-structure in $\mathcal{D}$, where
\begin{equation}\nonumber
\mathcal{D}^{\leq 0}_{(\mathcal{T},\mathcal{F})}=\left\{E\in\mathcal{D}: H^i(E)=0,\ \forall i>0, H^0(E)\in \mathcal{T}\right\}
\end{equation}
\begin{equation}\nonumber
\mathcal{D}^{\geq 0}_{(\mathcal{T},\mathcal{F})}=\left\{E\in\mathcal{D}: H^i(E)=0,\ \forall i<-1, H^{-1}(E)\in \mathcal{F}\right\}.
\end{equation}
\end{theorem}

The t-structure $(\mathcal{D}^{\leq 0}_{(\mathcal{T},\mathcal{F})}, \mathcal{D}^{\geq 0}_{(\mathcal{T},\mathcal{F})})$ is called \textbf{the HRS-tilt} of $(\Dcal^{\leq 0},\Dcal^{\geq 0})$ with respect to $(\Tcal,\Fcal)$. 

\begin{remark}\label{inter}
HRS-tilts span an important class of t-structures. In fact, it is known from \cite[theorem 3.1]{BeligiannisReiten07} and \cite[proposition 2.1]{Woolf10} that, for a bounded t-structure $(\Dcal^{\leq 0},\Dcal^{\geq 0})$, the bounded t-structures $(\Ccal^{\leq 0},\Ccal^{\geq 0})$ such that $\Dcal^{\leq -1}\subseteq \Ccal^{\leq 0}\subseteq \Dcal^{\leq 0}$ are precisely those obtained from $(\Dcal^{\leq 0},\Dcal^{\geq 0})$ by an HRS-tilt with respect to a torsion pair.
\end{remark}

Assume that  the heart $\Acal$ of the t-structure $(\Dcal^{\leq 0},\Dcal^{\geq 0})$ is a length category and let $S$ be a simple object of $\Acal$ without self-extensions. The left mutation of $(\Dcal^{\leq 0},\Dcal^{\geq 0})$ with respect to $S$ is defined by
\begin{equation}\nonumber
\mu^-_S(\Dcal^{\leq 0},\Dcal^{\geq 0})=(\Dcal^{\leq 0}_{-,S},\Dcal^{\geq 0}_{-,S}):=(\Dcal^{\leq 0}_{(\ ^\perp S,add(S))}, \Dcal^{\geq 0}_{(\ ^\perp S,add(S))})
\end{equation}
and its right mutation as
\begin{equation}\nonumber
\mu^+_S(\Dcal^{\leq 0},\Dcal^{\geq 0})=(\Dcal^{\leq 0}_{+,S},\Dcal^{\geq 0}_{+,S}):=(\Dcal^{\leq 0}_{(add(S), S^\perp)}[-1], \Dcal^{\geq 0}_{(add(S),S^\perp)}[-1])
\end{equation}
The torsion pairs $(add(S),S^\perp)$ and $(\ ^\perp S,add(S))$ will be called \textbf{mutation torsion pairs}.

On the other hand, silting mutation was introduced and studied by Buan, Reiten and Thomas in \cite{BuanReitenThomas11} and, independently, by Aihara and Iyama in \cite{AiharaIyama12}. We recall its definition.

\begin{definition}
Let $M=X\oplus Y$ be a silting object  in $\Dcal^b(R)$. The \textbf{left mutation of $M$ at $X$}, denoted by $\mu^-_X(M)$, is defined as the direct sum $\tilde{X}\oplus Y$, where $\tilde{X}$ is the cone of a left $add(Y)$-approximation of $X$ (i.e., of a morphism $\phi:X\ra L$, with $L$ in $add(Y)$ such that for any $Z$ in $add(Y)$, $\Hom_{\Dcal^b(R)}(\phi, Z)$ is surjective). 

Similarly, the \textbf{right mutation of $M$ at $X$}, denoted by $\mu^+_X(M)$, is defined as the sum $\bar{X}\oplus Y$, where $\bar{X}$ is the cone of a right $add(Y)$-approximation of $X$ (i.e., of a morphism $\psi:K\ra X$, with $K$ in $add(Y)$ such that, for any $Z$ in $add(Y)$, $\Hom_{\Dcal^b(R)}(Z,\psi)$ is surjective). 

We say that a mutation is \textbf{irreducible} if $X$ is indecomposable.
\end{definition}

\begin{theorem}\emph{(\cite[theorem 7.12]{KoenigYang12}).}\label{mutations commute} Irreducible silting mutations are compatible with the bijections established in theorem \ref{bijections}. More precisely, if $S_1,...,S_n$ are the simple objects in the heart of a t-structure $(\Dcal^{\leq 0},\Dcal^{\geq 0})$ and $X_1,..., X_n$ are indecomposable Ext-projective objects in $\Dcal^{\leq 0}$ such that $\Hom(X_i,S_i)\neq 0$, for all $1\leq i\leq n$, then the aisles corresponding to the silting objects $\mu^+_{X_i}(M)$ and $\mu_{X_i}^-(M)$ are $\Dcal^{\leq 0}_{+,S_i}$ and $\Dcal^{\leq 0}_{-,S_i}$ respectively.
\end{theorem}
The question of whether these mutations are compatible with glueing is discussed in section 6.

\section{Glueing of silting objects}

The results in this section make explicit in a more general setting the bijection established by Aihara and Iyama 
in \cite[theorem 2.37]{AiharaIyama12}, for Krull-Schmidt triangulated categories $\Dcal$, between the equivalence classes of silting objects in $\Dcal$ containing $i_*Y$ as a direct summand and the equivalence classes of silting objects in the triangle quotient $\Dcal/\thick(i_*Y)$. Here $\thick(i_*Y)$ is the smallest thick triangulated subcategory of $\Dcal$ containing $i_*Y$. More precisely, the theorem below shows how to glue two silting objects of $\Xcal$ and $\Ycal$ into a silting object of $\Dcal$ with respect to a recollement of the form (\ref{recollement}) and compatibly with the bijection between silting objects and co-t-structures of theorem \ref{general bijection}.  

\begin{theorem}\label{glue silting}
Let $\Rcal$ be a recollement of $\Dcal$ of the form (\ref{recollement}). Let $X$ and $Y$ be silting objects corresponding to co-t-structures $(\Xcal_{\geq 0},\Xcal_{\leq 0})$ and $(\Ycal_{\geq 0},\Ycal_{\leq 0})$ in $\Xcal$ and $\Ycal$, respectively. Then the induced co-t-structure $(\Dcal_{\geq 0},\Dcal_{\leq 0})$ in $\Dcal$ is associated with the silting object $Z=i_*Y\oplus K_X$, with $K_X$ defined by the following triangle
\begin{equation}\nonumber
i_*\beta_{\geq 1}i^!j_!X\longrightarrow j_!X\longrightarrow K_X\longrightarrow (i_*\beta_{\geq 1}i^!j_!X)[1],
\end{equation}
where $\beta_{\geq 1}$ is a (non-functorial) choice of truncation for the co-t-structure $(\Ycal_{\geq 0},\Ycal_{\leq 0})$ in $\Ycal$.
\end{theorem}
\begin{proof}
First observe that, since $i_*$ is a fully faithful functor, $i_*Y$ is partial silting and, moreover, it is easily checked to lie in the co-heart of the glued co-t-structure $(\Dcal_{\geq 0},\Dcal_{\leq 0})$ in $\Dcal$. 

Let us consider the following (non-functorial) triangle associated with the co-t-structure in $\Ycal$ 
\begin{equation}\nonumber
i_*\beta_{\geq 1}i^!j_!X\longrightarrow i_*i^!j_!X\longrightarrow i_*\beta_{\leq 0}i^!j_!X\longrightarrow (i_*\beta_{\geq 1}i^!j_!X)[1]
\end{equation}
and the following universal triangle of the recollement (applied to $j_!X$)
\begin{equation}\nonumber
i_*i^!j_!X\longrightarrow j_!X\longrightarrow j_*X\longrightarrow (i_*i^!j_!X)[1].
\end{equation}
Hence, we get the following commutative diagram where the rows are triangles
\begin{equation}\nonumber
\xymatrix{i_*\beta_{\geq 1}i^!j_!X\ar[r]\ar[d]&i_*i^!j_!X\ar[r]\ar[d]& i_*\beta_{\leq 0}i^!j_!X\ar[r]& (i_*\beta_{\geq 1}i^!j_!X)[1]\ar[d]\\ i_*\beta_{\geq 1}i^!j_!X\ar[r]\ar[d]&j_!X\ar[r]\ar[d]& K_X \ar[r]& (i_*\beta_{\geq 1}i^!j_!X)[1]\ar[d]\\ i_*i^!j_!X\ar[r]&j_!X\ar[r]&j_*X\ar[r]& (i_*i^!j_!X)[1].}
\end{equation}
By the octahedral axiom, this can be completed to a commutative diagram
\begin{equation}\nonumber
\xymatrix{i_*\beta_{\geq 1}i^!j_!X\ar[r]\ar[d]&i_*i^!j_!X\ar[r]\ar[d]& i_*\beta_{\leq 0}i^!j_!X\ar[r]\ar[d]& (i_*\beta_{\geq 1}i^!j_!X)[1]\ar[d]\\ i_*\beta_{\geq 1}i^!j_!X\ar[r]\ar[d]&j_!X\ar[r]\ar[d]& K_X \ar[r]\ar[d]& (i_*\beta_{\geq 1}i^!j_!X)[1]\ar[d]\\ i_*i^!j_!X\ar[r]\ar[d]&j_!X\ar[r]\ar[d]&j_*X\ar[r]\ar[d]& (i_*i^!j_!X)[1]\ar[d]\\ i_*\beta_{\leq 0}i^!j_!X\ar[r]&K_X\ar[r]&j_*X\ar[r]&(i_*\beta_{\leq 0}i^!j_!X)[1]}
\end{equation}
the rows of which are again triangles. We will now show that $K_X$ lies in the co-heart of $(\Dcal_{\geq 0},\Dcal_{\leq 0})$. Indeed, it is clear that $i_*\beta_{\geq 1}i^!j_!X\in \Dcal_{\geq 1}$ and, therefore, $(i_*\beta_{\geq 1}i^!j_!X)[1]$ lies in $\Dcal_{\geq 0}$. Clearly we also have that $j_!X$ lies in $\Dcal_{\geq 0}$ and, thus, the second row of the diagram shows that $K_X$ lies in $\Dcal_{\geq 0}$. Similarly, since $i_*\beta_{\leq 0}i^!j_!X$ lies in $\Dcal_{\leq 0}$ and $j_*X$ lies in $\Dcal_{\leq 0}$, it follows that $K_X$ lies in $\Dcal_{\leq 0}$, proving that it lies in the co-heart.

Finally, it is enough to observe that $Z$ generates $\Dcal$. It is clear that $i_*Y$ and $j_!X$ generate $\Dcal$ since $X$ and $Y$ generate $\Xcal$ and $\Ycal$ respectively and $i_*$ and $j_!$ are fully faithful functors. But, by the second triangle of the diagram, $j_!X$ can be generated by $K_X$ and an object in the image of $i_*$, which is generated by $i_*Y$. Therefore $Z$ generates $\Dcal$ and thus it is silting.
\end{proof}
\begin{remark}\label{r:the-morphism}
Note that the leftmost morphism $i_*\beta_{\geq 1}i^!j_!X\rightarrow j_!X$ of a triangle in theorem \ref{glue silting} corresponds to the chosen morphism $\beta_{\geq 1}i^!j_!X\rightarrow i^!j_!X$ via the adjunction morphism
\[\Hom(i_*\beta_{\geq 1}i^!j_!X,j_!X)\stackrel{\simeq}{\longrightarrow}\Hom(\beta_{\geq 1}i^!j_!X,i^!j_!X).\]
\end{remark}

\begin{remark}
In Krull-Schmidt categories (such as $\Dcal^b(R)$ for an algebra, finite dimensional over a field), the object $Z$ constructed in the proof of the theorem is not necessarily a basic object (i.e., the indecomposable summands of $Z$ can appear with multiplicities) even if both $X$ and $Y$ are basic. We can, however, obtain a basic silting object in the corresponding co-heart by ignoring the multiplicities of the direct summands of $Z$. 
\end{remark}

Similarly, we can describe a glueing of silting objects that is compatible with the bijection between silting objects and t-structures for $\Dcal=\Dcal^b(R)$, where $R$ has finite global dimension. 

\begin{corollary}\label{glue for t-str}
Let $\Rcal$ be a recollement of $\Dcal=\Dcal^b(R)$ of the form (\ref{recollement}), with $\Xcal=\Dcal^b(C)$ and $\Ycal=\Dcal^b(B)$. 
Let $X$ and $Y$ be silting objects corresponding to t-structures $(\Xcal^{\leq 0},\Xcal^{\geq 0})$ and $(\Ycal^{\leq 0},\Ycal^{\geq 0})$ in $\Xcal$ and $\Ycal$, respectively. Suppose that $R$ has finite global dimension. Then the glued t-structure $(\Dcal^{\leq 0},\Dcal^{\geq 0})$ in $\Dcal$ is associated with the silting object $Z=j_!X\oplus K_Y$, with $K_Y$ defined by the triangle
\begin{equation}\nonumber
j_!\alpha_{\geq 1}j^*i_\#Y\longrightarrow i_\#Y\longrightarrow K_Y\longrightarrow (j_!\alpha_{\geq 1}j^*i_\#Y)[1],
\end{equation}
where $\alpha_{\geq 1}$ is a (non-functorial) choice of truncation for the left adjacent co-t-structure of the t-structure $(\Xcal^{\leq 0},\Xcal^{\geq 0})$ in $\Xcal$.
\end{corollary}
\begin{proof}
This follows from the previous theorem by observing, as in remark \ref{adjoint structures} that $\Dcal_{\leq 0}^U=\Dcal^{\leq 0}$ and that the left adjacent co-t-structure of a t-structure corresponds to the same silting object. 
\end{proof}

The following definition settles what we will mean by \textbf{glueing silting objects}. As shown by theorem \ref{glue silting}, it is more natural, in this setting, to consider co-t-structures rather than t-structures. 

\begin{definition}
Let $\Dcal,\Xcal,\Ycal$ be triangulated categories and $\Rcal$ a recollement of the form (\ref{recollement}). We say that a silting object $Z\in\Dcal$ is \textbf{glued from $X\in\Xcal$ and $Y\in\Ycal$ with respect to $\Rcal$} if $Z$ is obtained by the construction of theorem \ref{glue silting}, i.e., $Z$ corresponds to the co-t-structure glued from the co-t-structures associated to $X$ and $Y$ with respect to $\Rcal$.
\end{definition} 

The object $K_X$ of theorem \ref{glue silting} can be described in a non-constructive way as follows.

\begin{proposition}\label{characterising glued silting}
Let $\Rcal$ be a recollement of a triangulated category $\Dcal$ of the form (\ref{recollement}). Let $X$ and $Y$ be silting objects corresponding to co-t-structures $(\Xcal_{\geq 0},\Xcal_{\leq 0})$ and $(\Ycal_{\geq 0},\Ycal_{\leq 0})$ in $\Xcal$ and $\Ycal$, respectively. Let $Z=i_*Y\oplus K_X$ be the silting object in $\Dcal$ glued from $X$ and $Y$ with respect to $\Rcal$, compatible with the glued co-t-structure $(\Dcal_{\geq 0},\Dcal_{\leq 0})$. Then the following holds.
\begin{enumerate}
\item $K_X$ is a right $\Dcal_{\geq 0}\cap\Dcal_{\leq 0}$-approximation of $j_*X$;
\item $K_X$ is a left $\Dcal_{\geq 0}\cap\Dcal_{\leq 0}$-approximation of $j_!X$;
\item Up to summands in $\Dcal_{\geq 0}\cap\Dcal_{\leq 0}$, $K_X$ is uniquely determined by the following conditions:
\begin{enumerate}
\item[(i)] $j^*K_X=X$;
\item[(ii)] $i^*K_X\in\Ycal_{\geq 0}$;
\item[(iii)] $i^!K_X\in \Ycal_{\leq 0}$.
\end{enumerate}
\end{enumerate}
\end{proposition}
\begin{proof}
To prove (1), let $C$ be an object of $\Dcal_{\geq 0}\cap\Dcal_{\leq 0}$ and $f$ a morphism in $\Hom_{\Dcal}(C,j_*X)$. Since $(i_*\beta_{\leq 0}i^!j_!X)[1]$ lies in $\Dcal_{\leq -1}$, it is clear that $\Hom_{\Dcal}(C,(i_*\beta_{\leq 0}i^!j_!X)[1])=0$ and thus, by using the defining triangle
\begin{equation}\nonumber
i_*\beta_{\leq 0}i^!j_!X\longrightarrow K_X\longrightarrow j_*X\longrightarrow (i_*\beta_{\leq 0}i^!j_!X)[1],
\end{equation}
 $f$ factors through $K_X$, proving (1). Analogously, (2) can be shown using the defining triangle
 \begin{equation}\nonumber
i_*\beta_{\geq 1}i^!j_!X\longrightarrow j_!X\longrightarrow K_X\longrightarrow (i_*\beta_{\geq 1}i^!j_!X)[1].
\end{equation}

In order to prove (3) observe that, for any object $L$ satisfying properties (i), (ii) and (iii),  we have a canonical triangle coming from the recollement $\Rcal$
\begin{equation}\nonumber
i_*i^!L\longrightarrow L\longrightarrow j_*X\longrightarrow i_*i^!L[1]
\end{equation}
and, since $i_*(\Ycal_{\leq 0})\subseteq \Dcal_{\leq 0}$, we have that $i_*i^!L[1]$ lies in $\Dcal_{\leq -1}$. Therefore, for any object $C$ in $\Dcal_{\geq 0}\cap\Dcal_{\leq 0}$, any map from $C$ to $j_*X$ factors through $L$, proving that $L$ is a right $\Dcal_{\geq 0}\cap\Dcal_{\leq 0}$-approximation of $j_*X$. Hence, by (1), $L$ differs from $K_X$ by a summand in $\Dcal_{\geq 0}\cap\Dcal_{\leq 0}$. 
\end{proof}

\section{Glueing of tilting objects}

In this section we investigate necessary and sufficient conditions for the glued silting to be tilting. These conditions will be expressed exclusively in terms of the functors of the recollement and of vanishing conditions on $\Xcal$ and $\Ycal$ rather than in $\Dcal$. In this section $R$ has finite global dimension and, thus, $\Dcal^b(R)$ has a Serre functor. Under this assumption we can use theorem \ref{reflect} and the functors therein.

\begin{proposition}\label{glue tilting}
Let $\Rcal$ be a recollement of $\Dcal=\Dcal^b(R)$ of the form (\ref{recollement}) with  $\Xcal=\Dcal^b(C)$ and $\Ycal=\Dcal^b(B)$. Let $X$ and $Y$ be silting objects of $\Xcal$ and $\Ycal$, respectively. Then $Z=i_*Y\oplus K_X$ is tilting if and only if the following conditions are satisfied
\begin{enumerate}
\item Y is tilting;
\item $\Hom_{\Ycal}(Y,i^!j_!X[k])=0$ for all $k<0$;
\item $\Hom_{\Ycal}(i^*j_*X,Y[k])=0$ for all $k<0$;
\item $\Hom_{\Xcal}(X, j^+K_X[k])=0$ for all $k<0$.
\end{enumerate}
\end{proposition}
\begin{proof}
The induced silting object $Z$ corresponds to the co-t-structure glued along $\Rcal$ and to the t-structure glued along $\Rcal_L$ from the respective structures on $\Xcal$ and $\Ycal$ associated with $X$ and $Y$ respectively. The statement that $Z$ is tilting is equivalent to the statement that $Z$ lies in the heart of $(\Dcal_L^{\leq 0}, \Dcal^{\geq 0}_L)$ which, by the description of the glued structures, translates into the conditions
\begin{itemize}
\item[(i)] $i^!Z\in \Ycal^{\leq 0}\cap\Ycal^{\geq 0}$ or, equivalently, $Y, i^!K_X\in\Ycal^{\leq 0}\cap\Ycal^{\geq 0}$;
\item[(ii)] $j^*Z\in\Xcal^{\leq 0}$;
\item[(iii)] $j^+Z\in\Xcal^{\geq 0}$.
\end{itemize}

To examine these conditions we will use the following triangle defining $K_X$ 
\begin{equation}\label{triang 2}
i_*\beta_{\leq 0}i^!j_!X\longrightarrow K_X\longrightarrow j_*X\longrightarrow (i_*\beta_{\leq 0}i^!j_!X)[1].
\end{equation}
Condition (i) is equivalent to $Y$ being tilting (corresponding to condition (1) of the theorem) and $i^!K_X$ lying in $\Ycal^{\leq 0}\cap\Ycal^{\geq 0}$. By applying $i^!$ to the triangle (\ref{triang 2}) we get that $i^!K_X\cong \beta_{\leq 0}i^!j_!X$ lies in $\Ycal_{\leq 0}=\Ycal^{\leq 0}$. By remark \ref{tilting}, $\beta_{\leq 0}$ corresponds to the stupid truncation in $\Dcal^b(End(Y))$ and, therefore, the fact that $i^!K_X$ lies in $\Ycal^{\leq 0}\cap\Ycal^{\geq 0}$ is equivalent to the condition that 
$i^!j_!X$ lies in $\Ycal^{\geq 0}$. This can then be translated to the condition (2) of the theorem, by definition of the t-structure associated with a silting object.

Clearly, since $j^*$ maps the co-heart in $\Dcal$ to the co-heart in $\Xcal$, (ii) is automatically satisfied for any silting $Z$ and, thus, this condition is irrelevant. 

Condition (iii) is equivalent to $\Hom_{\Xcal}(X,j^+Z[k])=\Hom_{\Dcal}(j_*X,Z[k])=0$ for all $k\leq 0$. By splitting $Z$ into its summands $i_*Y$ and $K_X$ and applying adjunction we get precisely conditions (3) and (4) of the theorem.
\end{proof}

\begin{remark}\label{two functors}
To simplify our notation, we will use the fact (see \cite[section 1.4]{BeilinsonBernsteinDeligne82}) that $i^*j_*=i^!j_![1]$. 
\end{remark}

In many important cases, the left side of the recollement is just given by $\Dcal^b(\Kbb)$.

\begin{corollary}\label{c:left-nece}
Let $\Rcal$ be a recollement of $\Dcal=\Dcal^b(R)$ of the form (\ref{recollement}), with $\Xcal=\Dcal^b(C)$ and $\Ycal=\Dcal^b(\Kbb)$. Let $X$ be a silting object in $\Xcal$ and $Y=\Kbb$. 
If $Z=i_*Y\oplus K_X$ is tilting then there are finite dimensional $\Kbb$-vector spaces $X'_{-1},X'_{0}$
 such that $i^*j_*X\cong X'_{-1}[1]\oplus X'_0$.
\end{corollary}
\begin{proof}
In view of remark~\ref{two functors}, the condition (2) of proposition~\ref{glue tilting} is equivalent to $H^k(i^*j_*X)=0$ for all $k<-1$, since $Y=B=\Kbb$. Since the Serre functor of $\Dcal^b(\Kbb)$ is isomorphic to the identity functor, condition (3) of proposition~\ref{glue tilting} is equivalent to $\Hom_\Ycal(Y[k],i^*j_*X)=0$ for all $k<0$, which is equivalent to $H^k(i^*j_*X)=0$ for all $k>0$. To summarise, if $Z$ is a tilting object, then it follows from proposition~\ref{glue tilting} that $H^k(i^*j_*X)=0$ for $k\neq -1,0$, thus proving the corollary.
\end{proof}

\begin{example}\label{example complete to tilting}
Let $R$ be the $\Kbb$-algebra given by the quiver with relations
\[\xymatrix{&&2\ar@<.7ex>[dd]^{\gamma}\ar@<-.7ex>[dd]_{\beta}\\
1\ar[rru]^{\alpha}&&&,\\
&& 3\ar[llu]^{\delta}}\hspace{10pt}\xymatrix{\\ \beta\alpha,~~ \alpha\delta,~~\delta\gamma.}\]
The simple module $S_1$ supported at $1$ is a partial tilting module of projective dimension $2$. It has a minimal projective resolution over $R$ given by the exact sequence
\[\xymatrix{0\ar[r] &P_2\ar[r]^{\beta} & P_3\ar[r]^{\delta} & {P_1}\ar[r] & S_1\ar[r]&0}.\]
Rickard and Schofield showed in~\cite{RickardSchofield89} that $S_1$ cannot be completed to a tilting module over $R$. We will strengthen this result by showing that $S_1$ cannot be completed to a tilting object in $\Dcal^b(R)$.

If $e$ is the idempotent $e_2+e_3$ then, as a right $A$-module, $R/ReR$ is isomorphic to $S_1$, which does not have self-extensions. It, therefore, follows from \cite[theorem 3.1]{ClineParshallScott88b} and \cite[section 2]{ClineParshallScott96} that there is a recollement
\begin{equation}\nonumber
\xymatrix@C=3pc{\Dcal^b(R/ReR)\ar[r]^(0.55){i_*}&\Dcal^b(R)\ar@<3ex>[l]_(0.45){i^!}\ar@<-3ex>[l]_(0.45){i^*}\ar[r]^{j^*}&\Dcal^b(eRe)\ar@<3ex>_{j_*}[l]\ar@<-3ex>_{j_!}[l]},
\end{equation}
where
\begin{eqnarray}\label{eq:six-functors}\begin{array}{ll}
i^*=-\lten_R R/ReR, & j_!=-\lten_{eRe} eR,\\
i_*=-\lten_{R/ReR}R/ReR, & j^*=-\lten_R Re,\\
i^!=\RHom_R(R/ReR,-),& j_*=\RHom_{eRe}(Re,-).
\end{array}
\end{eqnarray}
Fix $Y=R/ReR$ and let $X\in\Dcal^b(eRe)$ be a silting object (e.g. $X=eRe$). Then theorem~\ref{glue silting} yields a completion $Z=i_*Y\oplus K_X$ of the partial silting object $i_*Y=S_1$ into a silting object. However, as we will show, $Z$ is never tilting and, thus, $S_1$ cannot be completed to a tilting object.

Let $T$ be a basic tilting object of $\Dcal^b(R)$ which contains $S_1$ as a direct summand. The functor $j^*$ factors through the canonical projection functor $\pi:\Dcal^b(R)\rightarrow\Dcal^b(R)/\thick(S_1)$, as follows
\[\xymatrix@R=1pc{\Dcal^b(R)\ar[rr]^{j^*}\ar[rd]_(0.4){\pi} && \Dcal^b(eRe)\\
&\Dcal^b(R)/\thick(S_1).\ar[ru]_{\simeq}}\]
As $\thick(S_1)$ is a silting subcategory of $\Dcal^b(R)$, by~\cite[theorem 2.37]{AiharaIyama12}, the object $X=j^*(T)$ is silting in $\Dcal^b(eRe)$ and $T$ is equivalent to $S_1\oplus K_X$. Observe that $eRe$ is the path algebra of the Kronecker quiver
\[\xymatrix{2\ar@<.7ex>[rr]^{\gamma}\ar@<-.7ex>[rr]_{\beta} && 3}.\]
We draw the component of the AR quiver of $\Dcal^b(eRe)$ containing the preprojective modules
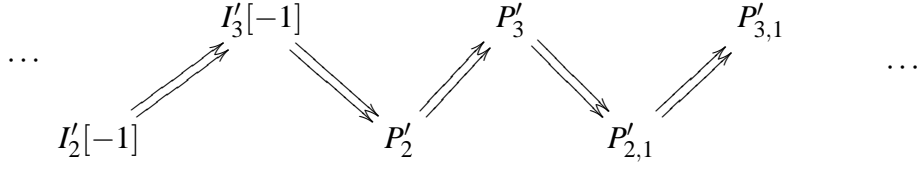
\[\xymatrix@R=0.5pc{\\ \cdots\\}\xymatrix{ &I_3'[-1]\ar@<.3ex>[rd]\ar@<-.3ex>[rd] && P_3'\ar@<.3ex>[rd]\ar@<-.3ex>[rd]&&P_{3,1}'&\\
I_2'[-1]\ar@<.3ex>[ru]\ar@<-.3ex>[ru]&&P_2'\ar@<.3ex>[ru]\ar@<-.3ex>[ru] && P_{2,1}'\ar@<.3ex>[ru]\ar@<-.3ex>[ru]}\xymatrix@R=0.7pc{\\ \cdots\\}\]
where $P_2'$ and $P_3'$ are respectively the indecomposable projective $eRe$-module at the vertex $2$ and $3$. Let $\tau=\tau_{eRe}$ denote the Auslander--Reiten translation of $\Dcal^b(eRe)$. Then according to~\cite[lemma 3.1]{SoutoTrepode11}, there exists $p$ and $q$ such that $\tau^pP_2'[q]$ is a direct summand of $K_X$. We claim that 
$$H^n(i^*j_*(\tau^p P_2'[q]))=\begin{cases} \mathbb{K} & \text{ if } n=-q-1,-q+1,\\ 0 & \text{ otherwise.}\end{cases}$$ As a consequence of corollary~\ref{c:left-nece}, $S_1\oplus K_X$ cannot be a tilting object, a contradiction.

The claim follows from a direct computation, which we show only for $P_{2,1}'=\tau^{-1}P_2'$ (it is analogous for others). 
For $n\in\mathbb{Z}$, we have 
\begin{eqnarray*}
H^n(i^*j_*(P_{2,1}'))\hspace{-7pt}&=&\hspace{-7pt}H^{n+1}(i^!j_!(P_{2,1}'))=H^{n+1}(\RHom_R(R/ReR,j_!(P_{2,1}')))\\
\hspace{-7pt}&=&\hspace{-7pt}H^{n+1}(\RHom_R(S_1,j_!(P_{2,1}')))=\Hom_{\Dcal^b(R)}(S_1,j_!(P_{2,1}')[n+1]).\end{eqnarray*}
We take a minimal projective resolution of $P_{2,1}'$
\[\xymatrix{P_2'\ar[r]^(0.35){{\gamma\choose\beta}}&\underline{P_3'\oplus P_3'}},\]
where the underlined term is in degree $0$.
Applying $j_!$ we obtain the following object in $\Dcal^b(R)$:
\[\xymatrix{P_2\ar[r]^(0.35){{\gamma\choose\beta}}&\underline{P_3\oplus P_3}}.\] 
The above differential is injective and its cokernel has the following minimal injective resolution
\[\xymatrix{\underline{I_2\oplus I_2\oplus I_1}\ar[r] & I_3\oplus I_3\oplus I_3\ar[r] & I_1\ar[r]& I_2\ar[r] & I_3}.\]
Therefore $j_!(P_{2,1}')$ is isomorphic to this complex in $\Dcal^b(R)$ and it follows that 
$$\Hom_R(S_1,j_!(P_{2,1}')[n+1])=\begin{cases} \mathbb{K} & \text{ if } n=-1,1,\\ 0 & \text{ otherwise.}\end{cases}$$

\end{example}

\bigskip

In using proposition \ref{glue tilting} we do not seem to be dealing with computations in $\Dcal$. Still, these occur in the construction of $K_X$. This can be avoided by introducing the assumption that $X$ is tilting.

\begin{theorem}\label{conditions tilting}
Let $\Rcal$ be a recollement of $\Dcal=\Dcal^b(R)$ of the form (\ref{recollement}) with  $\Xcal=\Dcal^b(C)$ and $\Ycal=\Dcal^b(B)$. Let $X$ and $Y$ be tilting objects of $\Xcal$ and $\Ycal$, respectively. Then $Z=i_*Y\oplus K_X$ is tilting in $\Dcal$ if and only if the following conditions are satisfied
\begin{enumerate}
\item[(a)] $\Hom_{\Ycal}(Y,i^*j_*X[k])=0$ for all $k<-1$;
\item[(b)] $\Hom_{\Ycal}(i^*j_*X,Y[k])=0$ for all $k<0$;
\item[(c)] $\Hom_{\Ycal}(i^*j_*X,i^*j_*X[k])=0$ for all $k<-1$.
\end{enumerate}
\end{theorem}
\begin{proof}
Condition $(a)$ is obtained from condition (2) of proposition \ref{glue tilting} and from remark \ref{two functors}, while condition (b) corresponds exactly to condition (3) of proposition \ref{glue tilting}. We will show that if $X$ is tilting and (a) and (b) are satisfied then condition (4) of proposition \ref{glue tilting} is equivalent to (c). Indeed, if $X$ is tilting, $\Hom_\Dcal(j_*X,j_*X[k])=0$ for all $k\neq 0$ and thus, applying the functor $\Hom_\Dcal(j_*X,\_)$ to the $k$-th shift of triangle (\ref{triang 2}), we get $\Hom_{\Dcal}(j_*X,K_X[k])\cong\Hom_\Dcal(j_*X,(i_*\beta_{\leq 0}i^!j_!X)[k])$ for all $k\neq 0,1$. Since $(j_*,j^+)$ is an adjoint pair, this shows that condition (4) of proposition \ref{glue tilting} is equivalent to 
\begin{equation}\nonumber
\Hom_\Dcal(j_*X,(i_*\beta_{\leq 0}i^!j_!X)[k])=\Hom_{\Dcal}(i^*j_*X,(\beta_{\leq 0}i^!j_!X)[k])=0, \ \forall k<0.
\end{equation} 
Applying the functor $\Hom_{\Ycal}(i^*j_*X,\_)$ to the triangle
\begin{equation}\nonumber
(\beta_{\geq 1}i^!j_!X)[k]\longrightarrow i^!j_!X[k]\longrightarrow (\beta_{\leq 0}i^!j_!X)[k]\longrightarrow (\beta_{\geq 1}i^!j_!X)[k+1]
\end{equation}
we get that $\Hom_{\Ycal}(i^*j_*X, \beta_{\leq 0}i^!j_!X[k])\cong\Hom_{\Ycal}(i^*j_*X, i^!j_!X[k])$, for all $k<0$. Indeed, this follows from (b) 
after recalling that $\beta_{\geq 1}i^!j_!X$ lies in $\Ycal_{\geq 1}$, which by construction (see \ref{general bijection}) is suitably generated by $(Y[n])_{n<0}$ 
- showing that $\Hom_{\Ycal}(i^*j_*X,(\beta_{\geq 1}i^!j_!X)[k+1])=0$ for all $k<0$. Using again remark \ref{two functors}, we obtain condition (c), thus finishing the proof.
\end{proof}

Before showing applications of this theorem, we discuss the behaviour of the  Serre functor with respect to a t-structure corresponding to a tilting object. Recall from \cite[proposition 3.1.10]{BeilinsonBernsteinDeligne82} that for any bounded t-structure $(\Dcal^{\leq 0},\Dcal^{\geq 0})$ in $\Dcal^b(R)$ with heart $\Acal$, there is a triangle functor
$real: \Dcal^b(\Acal)\rightarrow \Dcal^b(R)$,
t-exact for the standard t-structure $(\Dcal^{\leq 0}_{st},\Dcal^{\geq 0}_{st})$ in $\Dcal^b(\Acal)$ and $(\Dcal^{\leq 0},\Dcal^{\geq 0})$ in $\Dcal^b(R)$, i.e., it is both right t-exact ($real(\Dcal^{\leq 0}_{st})\subseteq \Dcal^{\leq 0}$) and left t-exact ($real(\Dcal^{\geq 0}_{st})\subseteq\Dcal^{\geq 0}$).

\begin{lemma}\label{Serre functor}
Let $(\Dcal^{\leq 0},\Dcal^{\geq 0})$ be a bounded t-structure in $\Dcal^b(R)$ with associated basic silting object $T$. Let $S_R$ be the Serre functor of $\Dcal^b(R)$. Then the following are equivalent:
\begin{enumerate}
\item $T$ is tilting; 
\item $S_R$ is right t-exact;
\item $S_R(T)$ lies in the heart.
\end{enumerate}
\end{lemma}
\begin{proof}
$(1)\Rightarrow (2)$: Suppose $T$ is tilting. Let $\Gamma =$End$_{\Dcal^b(R)}(T)$ and let $(\Dcal^{\leq 0}_{st},\Dcal^{\geq 0}_{st})$ be the standard t-structure in $\Dcal^b(\Gamma)$. Then, by  \cite[theorem 6.6]{Alonso-Jeremias-Souto03}, the realisation functor is a t-exact equivalence. Therefore, $real \circ S_\Gamma = S_R\circ real$ (since Serre functors are unique) and $real(\Dcal_{st}^{\leq 0}) = \Dcal^{\leq 0}$. Note also that $S_\Gamma$ preserves the standard aisle in $\Dcal^b(\Gamma)$ since it can be realised as a derived tensor product with a bimodule. This is enough to observe that $S_R$ is right t-exact, since
\begin{equation}\nonumber
S_R(\Dcal^{\leq 0})=S_R(real(\Dcal_{st}^{\leq 0})) = real(S_\Gamma(\Dcal_{st}^{\leq 0})) \subseteq real(\Dcal^{\leq 0}_{st}) =\Dcal^{\leq 0}.
\end{equation}

$(2)\Rightarrow (3)$: For any $K\in\Dcal^{\leq -1}$, $\Hom_{\Dcal^b(R)}(K,S_RT)=\Hom_{\Dcal^b(R)}(T,K)=0$, by definition of the t-structure associated to $T$ (see subsection 2.3). This shows that $S_R T\in \Dcal^{\geq 0}$. If $S_R$ is right t-exact, then $S_R(T)\in\Dcal^{\leq 0}$ and, thus, $S_R(T)$ lies in the heart.

$(3)\Rightarrow (1)$: Suppose $S_R(T)$ lies in the heart. By the definition of the t-structure associated to $T$, it is clear that $\Hom(T,T[k])=\Hom(T,S_R(T)[-k])=0$ for all $k\neq 0$. Therefore $T$ is tilting.
\end{proof}

\begin{proposition}\label{sufficient condition}
Let $\Rcal$ be a recollement of $\Dcal=\Dcal^b(R)$ of the form (\ref{recollement}) with $\Xcal=\Dcal^b(C)$ and $\Ycal=\Dcal^b(B)$. Let $X$ and $Y$ be tilting objects of $\Xcal$ and $\Ycal$, respectively. If $i^*j_*X$ is an element of $\Dcal^b(B)$ such that 
$\Hom_\Ycal(Y,i^*j_*X[k])$ is zero except for two consecutive values of $k\in\Zbb$, then the family of silting objects $(Z_n=i_*Y\oplus K_{X[n]})_{n\in\Zbb}$ contains at least one tilting object.
\end{proposition}
\begin{proof}
Let $(\Ycal^{\leq 0},\Ycal^{\geq 0})$ be the t-structure associated with $Y$. By definition of the t-structure associated to $Y$, $\Hom_\Ycal(Y,i^*j_*X[k])$ is nonzero exactly for two consecutive values of $k\in\Zbb$ if and only if there is $a$ in $\Zbb$ such that  $i^*j_*X\in\Ycal^{\geq a}\cap\Ycal^{\leq a+1}$.  Since $Y$ is tilting, by lemma \ref{Serre functor} we have that $S_\Ycal i^*j_*X$ lies in $\Ycal^{\leq a+1}$. In theorem \ref{conditions tilting}, condition (a) can be reformulated as $i^*j_*X\in\Ycal^{\geq -1}$ and condition (b) as $S_\Ycal i^*j_*X\in\Ycal^{\leq 0}$. Note that, once condition (c) is satisfied for some $X[n]$ it is satisfied for all its shifts. By theorem \ref{conditions tilting}, $Z_n$ is tilting if and only if 
\begin{itemize}
\item[(a)] $i^*j_*X[n]\in\Ycal^{\geq -1}$, which means that $a-n\geq -1$ or, equivalently, $n\leq a+1$;
\item[(b)] $S_\Ycal i^*j_*X[n]\in\Ycal^{\leq 0}$ which is guaranteed if $a+1-n\leq 0$ or, equivalently, $n\geq a+1$ 
\item[(c)] $\Hom_\Ycal(i^*j_*X, i^*j_*X[k])=0$ for all $k<-1$.
\end{itemize}
Therefore, take $n=a+1$ and we only need to check condition (c). This condition is, however, automatic from our assumption that $i^*j_*X$ has cohomologies with respect to $(\Ycal^{\leq 0},\Ycal^{\geq 0})$ concentrated in exactly two consecutive degrees, thus finishing the proof.
\end{proof}

\begin{remark}
In the family $(Z_n)_{n\in\mathbb{Z}}$ there may be more than one tilting object. Let $a$ be the maximal integer such that $i^*j_*X\in\Ycal^{\geq a}$ and $b$ be the minimal integer such that $S_\Ycal i^*j_*X\in\Ycal^{\leq b}$. Then $Z_n$ is tilting if and only if $b\leq n \leq a+1$.
\end{remark}

From the proof above, we see that, for $X$ and $Y$ tilting, whenever $i^*j_*X$ lies in $\Ycal^{\leq 0}
\cap \Ycal^{\geq -1}$ (i.e., the cohomologies of $i^*j_*X$ with respect to $Y$ lie in degrees $-1$ and $0$), the induced silting
$Z=i_*Y\oplus K_X$ is tilting. Since $i^*j_*=i^!j_![1]$, the
condition is equivalent to that $i^!j_!X\in\Ycal^{\leq 1}\cap
\Ycal^{\geq 0}$ which, in its turn, is equivalent to $\Hom_\Ycal(Y,i^!j_!X[k])\cong\Hom_\Dcal(i_*Y,j_!X[k])=0$ whenever $k\neq 0,1$. Under this assumption, we can apply the construction from
\cite{AngeleriKoenigLiu11a} (theorem 2.5) to the pair $T_1=j_!X$ and $T_2=i_*Y$.  A similar construction was studied in \cite{Ladkani11} in the setting of triangular matrix rings.

We briefly recall the construction in \cite{AngeleriKoenigLiu11a}. Since $(j_!,j^*)$ is an adjoint pair in the recollement (of the form (\ref{recollement})),
 $\Hom_\Dcal(j_!X,i_*Y[k]) \cong \Hom_\Xcal(X,j^*i_*Y[k])=0$ for all $k\in\Z$. Let $m$
be the dimension of $\Hom_\Dcal(i_*Y,j_!X[1])$ and take a basis $\alpha_1,\ldots,\alpha_m:i_*Y\ra j_!X[1]$.
Consider the universal maps
\begin{eqnarray*}
\alpha=(\alpha_1,\ldots,\alpha_m)^{tr}&:&i_*Y^{\oplus m}\ra j_!X[1],\\
\beta=(\alpha_1,\ldots,\alpha_m)&:&i_*Y\ra j_!X[1]^{\oplus m}.\end{eqnarray*} The map $\alpha$ is left-universal, i.e., the induced map $$\Hom(i_*Y,\alpha):\Hom_\Dcal(i_*Y,i_*Y^{\oplus m}) \ra 
\Hom_\Dcal(i_*Y,j_!X[1])$$
is surjective. Similarly, the map $\beta$ is right-universal, i.e., the induced map $$\Hom(\beta,j_!X[1]):\Hom_\Dcal(j_!X[1]^{\oplus m},j_!X[1])\ra 
\Hom_\Dcal(i_*Y,j_!X[1])$$
is surjective. Consider the triangles determined by $\alpha$ and $\beta$:
\begin{eqnarray*}T_\alpha &:&  j_!X \ra C_1 \ra i_*Y^{\oplus m} \xrightarrow{\alpha} j_!X[1],\\
T_\beta &:&  j_!X^{\oplus m} \ra C_2 \ra i_*Y \xrightarrow{\beta} j_!X[1]^{\oplus m}.\end{eqnarray*}
Theorem 2.5 in \cite{AngeleriKoenigLiu11a} asserts that $i_*Y\oplus C_1$ and $j_!X\oplus C_2$ are tilting object in $\Dcal$.

The next proposition proves that the tilting object $i_*Y\oplus C_1$ is precisely the silting (indeed tilting) glued from $X\in\Xcal$ and $Y\in\Ycal$ 
with respect to the recollement $\Rcal$, and $j_!X\oplus C_2$ is the silting (indeed tilting) glued from $X\in\Xcal$ and $Y\in\Ycal$
with respect to the upper reflection $\Rcal_U$. As a consequence, under the conditions of theorem \ref{conditions tilting}, we are not only
able to glue tilting objects in the original recollement $\Rcal$, but also in the upper reflection $\Rcal_U$.

\begin{proposition}
Let $\Rcal$ be a recollement of $\Dcal=\Dcal^b(R)$ of the form (\ref{recollement}), with $\Xcal=\Dcal^b(C)$ and $\Ycal=\Dcal^b(B)$. Let $X$ and $Y$ be tilting objects of $\Xcal$ and $\Ycal$, respectively, such that  
$\Hom_\Ycal(Y,i^*j_*X[k])=0$ for all $k\neq 0,-1$. Then, up to multiplicity,
the tilting object $i_*Y\oplus C_1$ 
coincides with $Z=i_*Y\oplus K_X$ of theorem \ref{glue silting}, and the tilting object
$j_!X\oplus C_2$ coincides with $Z_U=j_!X\oplus K_Y$ of corollary \ref{glue for t-str}.
\end{proposition}

\begin{proof} By hypothesis, $i^*j_*X$ lies in $\Ycal^{\leq 0}\cap\Ycal^{\geq -1}$, for the t-structure $(\Ycal^{\leq 0},\Ycal^{\geq 0})$ associated with $Y$.
We will prove that $C_1$ satisfies the conditions (i) -- (iii) for $K_X$ in proposition \ref{characterising glued silting} (3).
It follows then $i_*Y\oplus C_1$ lies in the glued co-heart of $\Dcal$ and it generates $i_*Y\oplus 
j_!X$, thus generating the whole of $\Dcal$.
Therefore, up to multiplicities, it coincides with the glued silting $Z=i_*Y\oplus K_X$. 

For condition (i), we apply $j^*$ to the defining triangle
$T_\alpha$ of $C_1$ and obtain that $j^*C_1 = j^*j_!X = X$, since $j^*i_*=0$. For condition (ii), we apply $i^*$ to $T_\alpha$
and obtain that $i^*C_1 = i^*i_*Y^{\oplus m} = Y^{\oplus m} \in \Ycal_{\geq 0}$, since $i^*j_!=0$. For condition (iii) we apply
$i^!$ to $T_\alpha$ and obtain a triangle
$$i^!j_!X \ra i^!C_1 \ra i^!i_*Y^{\oplus m} \xrightarrow{i^!\alpha} i^!j_!X[1].$$
By applying $\Hom_\Ycal(Y,-)$ to this triangle, we obtain a long exact sequence
$$... \ra \Hom_\Ycal(Y,i^!i_*Y^{\oplus m}) \xrightarrow{(i^!\alpha)_*} \Hom_\Ycal(Y,i^!j_!X[1]) \ra \Hom_\Ycal(Y,i^!C_1[1]) \ra
\Hom_\Ycal(Y,i^!i_*Y[1]^{\oplus m}) \ra ... $$
For $k>0$, we have $\Hom_\Ycal(Y,i^!i_*Y[k]^{\oplus m})\simeq \Hom_\Ycal(Y,Y[k]^{\oplus m}) = 0$ (since, by assumption, $Y$ is tilting) and
for $k>1$, $\Hom_\Ycal(Y,i^!j_!X[k])=0$ (by our assumption on $i^*j_*X=i^!j_!X[1]$). 
Hence, $\Hom_\Ycal(Y,i^!C_1[k])=0$, for all $k>1$. By adjunction there is a commutative diagram
\[\xymatrix{(i^!\alpha)_*: \hspace{-15pt}& \Hom_\Ycal(Y,i^!i_*Y^{\oplus m}) \ar[r]& \Hom_\Ycal(Y,i^!j_!X[1])\\
\alpha_*:\hspace{-27pt}&\Hom_\Dcal(i_*Y,i_*Y^{\oplus m})\ar[u]^\simeq \ar[r]& \Hom_\Dcal(i_*Y,j_!X[1]).\ar[u]^\simeq}\]  Since, by construction, $\alpha$ 
is left-universal, i.e., $\alpha_*$ is surjective, it follows that $(i^!\alpha)_*$ is surjective. Hence, we get $\Hom_\Ycal(Y,i^!C_1[1])=0$ and, thus,  
$i^!C_1$ lies in $\Ycal^{\leq 0} = \Ycal_{\leq 0}$. The proof for statement about $j_!X\oplus C_2$ is analogous.
\end{proof}

\begin{section}{The hereditary case}
In this section we assume that $R$ has finite global dimension and we apply theorem \ref{conditions tilting} when $B$ is a hereditary algebra. The conditions for the glued silting to be tilting are then easier to handle.
\begin{proposition}
Let $\Rcal$ be a recollement of  $\Dcal=\Dcal^b(R)$ of the form (\ref{recollement}), with  $\Xcal=\Dcal^b(C)$, $\Ycal=\Dcal^b(B)$ and $B$ hereditary.
Let $X$ be a tilting object in $\Xcal$ and $Y=B$. Then $Z=i_*Y\oplus K_X$ is tilting in $\Dcal$ if and only if there are finitely generated $B$-modules $X'_{-1}$, $X'_{0}$ and $X'_{1}$ such that
\begin{itemize}
\item $i^*j_*X$ is isomorphic to $X'_{-1}[1]\oplus X'_0\oplus X'_1[-1]$;
\item $X'_1$ is either zero or not projective; 
\item $\Hom_B(X'_{1},X'_{-1})=0$.
\end{itemize}
\end{proposition}
\begin{proof}
We analyse the conditions of theorem \ref{conditions tilting} when $B$ is  hereditary and $Y=B$.
Since $B$ is hereditary, we may assume that $i^*j_*X=\bigoplus_{n\in\mathbb{Z}} X'_n[-n]$, with $X'_n$ a finitely generated right $B$-module for all $n$ in $\Zbb$.
In this setting, $\Hom_{\Ycal}(Y,i^*j_*X[k])=0, \text{ for all } k<-1$, is equivalent to
$$\Hom_{\Dcal^b(B)}(B, \bigoplus_{n\in\mathbb{Z}} X'_n[-n+k])=0 \text{ for all } k<-1$$
which happens if and only if $X'_{k}=0$, for all $k<-1$.
On the other hand, the condition
\[\Hom_{\Ycal}(i^*j_*X,Y[k])=\Hom_{\Dcal^b(B)}(i^*j_*X,B[k])=0, \text{ for all } k<0\] 
can be reformulated, using the Serre functor $S_B$ in $\Dcal^b(B)$, by
\[\Hom_{\Dcal^b(B)}(B,S_B(i^*j_*X)[k])=0, \text{ for all } k>0.\]
It is well known that $S_B=[1]\circ\tau_B$, where $\tau_B$ is the Auslander-Reiten translation in $\Dcal^b(B)$ and, thus, the statement above is equivalent to
\[ \Hom_{\Dcal^b(B)}(B,\bigoplus_{n\in\mathbb{Z}}\tau_B(X_n')[-n+k+1])=0,\ \forall k>0,\]
i.e., $X'_{k}=0, \text{ for all } k>1, \text{ and } X'_1 \text{ is not a projective } B \text{-module.}$ So, (a) and (b) of theorem~\ref{conditions tilting} hold if and only if
$i^*j_*X=X'_{-1}[1]\oplus X'_0\oplus X'_1[-1], \text{ with } X'_1 \text{ not projective as a $B$-module.}$
Assuming this, we can easily unfold the last condition of theorem \ref{conditions tilting} as follows
\begin{eqnarray*}
\Hom_{\Ycal}(i^*j_*X,i^*j_*X[k])=0, \forall k<-1 &\Leftrightarrow & \Hom_{\Dcal^b(B)}(X'_1,X'_{-1})=0, \end{eqnarray*}
thus finishing the proof. 
\end{proof}

As a nice corollary, we obtain a necessary and sufficient condition for the case $\Ycal=\Dcal^b(\Kbb)$.

\begin{corollary}\label{c:left=k}
Let $\Rcal$ be a recollement of $\Dcal=\Dcal^b(R)$ of the form (\ref{recollement}), with $\Xcal=\Dcal^b(C)$ and $\Ycal=\Dcal^b(\Kbb)$. Let $X$ be a tilting object in $\Xcal$ and $Y=\Kbb$. 
Then $Z=i_*Y\oplus K_X$ is tilting if and only if there are finite dimensional $\Kbb$-vector spaces $X'_{-1},X'_{0}$
 such that $i^*j_*X\cong X'_{-1}[1]\oplus X'_0$.
\end{corollary}

\begin{example}\label{ex:hereditary} Let $R$ be the path algebra over $\Kbb$ of the quiver $\xymatrix{1\ar[r] & 2,}$ of type $A_2$. Let $e=e_1$ be the trivial path at the vertex $1$. Consider the following standard recollement
\begin{equation}\nonumber
\xymatrix@C=3pc{\Dcal^b(R/ReR)\ar[r]^(0.55){i_*}&\Dcal^b(R)\ar@<3ex>[l]_(0.45){i^!}\ar@<-3ex>[l]_(0.45){i^*}\ar[r]^{j^*}&\Dcal^b(eRe)\ar@<3ex>_{j_*}[l]\ar@<-3ex>_{j_!}[l]},
\end{equation}
where the six functors are defined as in (\ref{eq:six-functors}).
Let $D$ denote the functor $\Hom_\Kbb(-,\Kbb)$. It is easy to observe that:
\begin{enumerate}
\item as algebras both $R/ReR$ and $eRe$ are isomorphic to $\Kbb$;
\item as an $R$-module $eR$ is the simple module $S_1$, supported at the vertex $1$, and has an injective resolution $(D(Re)\rightarrow D(R(1-e)))$;
\item as an $R$-module $R/ReR$ is the simple module $S_2$ supported at the vertex $2$.
\end{enumerate}
Fix $Y=R/ReR$ and $X=eRe$. Then, in $\Dcal^b(R)$, we have
\begin{eqnarray*}i^!j_!X &=& i^!(eR)\hspace{7pt}=\hspace{7pt}\RHom_R(R/ReR,eR)\\
&\cong&(\Hom_R(R/ReR,D(Re))\rightarrow\Hom_R(R/ReR,D(R(1-e))))\\
&\cong&R/ReR[-1].\end{eqnarray*}
By corollary~\ref{c:left=k}, there are exactly two tilting objects in $\{Z_n=i_*Y\oplus K_{X[n]}:n\in\mathbb{Z}\}$. We construct this family using theorem~\ref{glue silting}. For $n\in\mathbb{Z}$, we can choose the map $\beta_{\geq 1}(i^!j_!X[n])\rightarrow i^!j_!X[n]$ to be
\[\begin{cases} R/ReR[n-1]\stackrel{id}{\rightarrow} R/ReR[n-1] & \text{ if } n\leq 0,\\
0\rightarrow R/ReR[n-1] & \text{ if } n>0.\end{cases}\]
By adjunction we have a bijection
\[\Hom(i_*\beta_{\geq 1}(i^!j_!X[n]),j_!X[n])\stackrel{\simeq}{\longrightarrow}\Hom(\beta_{\geq 1}(i^!j_!X[n]),i^!j_!X[n]).\]
If $n\geq 0$, this adjunction morphism is
\[\Hom_{\Dcal^b(R)}(S_2[n-1],S_1[n])\stackrel{\simeq}{\longrightarrow}\Hom_{\Dcal^b(R)}(R/ReR[n-1],R/ReR[n-1]).\]
Both these vector spaces are $1$-dimensional over $\mathbb{K}$. Let $f$ be the preimage of $id_{R/ReR[n-1]}$. Then the cone of $f[-n]$ is $(1-e)R$, the projective cover $P_2$ of $S_2$. If $n<0$, the above adjunction morphism is
\[\Hom_{\Dcal^b(R)}(0,S_1[n])\stackrel{\simeq}{\longrightarrow}\Hom_{\Dcal^b(R)}(0,R/ReR[n-1]).\] Thus a morphism $i_*\beta_{\geq 1}(i^!j_!X[n])\rightarrow j_!X[n]$ as in theorem~\ref{glue silting} is (see remark~\ref{r:the-morphism})
\[\begin{cases} S_2[n-1]\stackrel{f}{\rightarrow} S_1[n] & \text{ if } n\leq 0,\\
0\rightarrow S_1[n] & \text{ if } n>0.\end{cases}\]
Hence, we have
\[K_{X[n]}=\begin{cases} P_2[n] & \text{ if } n\leq 0,\\
S_1[n] & \text{ if } n>0,\end{cases}\ \ \ \ \text{ and }\ \ \ \ \ Z_{n}=\begin{cases} S_2\oplus P_2[n] & \text{ if } n\leq 0,\\
S_2\oplus S_1[n] & \text{ if } n>0.\end{cases}\]
Among these silting objects, $Z_0=S_2\oplus P_2$ and $Z_1=S_2\oplus S_1[1]$ are tilting objects.
\end{example}

\end{section}

\bigskip

\section{HRS-tilts and recollements}

In this section we show that HRS-tilts of t-structures with respect to torsion theories (\cite{HappelReitenSmaloe96})
are \textit{compatible} with the glueing of t-structures via
recollements. The main results of this section are theorem \ref{t:glueing-vs-tilt} and proposition \ref{p:glueing-torsion-pair}.
Our notation is fixed as follows.
\begin{itemize}
\item $\Dcal$ is a triangulated category
admitting a recollement $\Rcal$ of the form (\ref{recollement});
\item 
$(\Ycal^{\leq 0},\Ycal^{\geq 0})$ and $(\Xcal^{\leq 0},\Xcal^{\geq
0})$ are bounded t-structures in $\Ycal$ and $\Xcal$ respectively; ,
\item $(\Dcal^{\leq 0},\Dcal^{\geq 0})$ is glued from
$(\Ycal^{\leq 0},\Ycal^{\geq 0})$ and $(\Xcal^{\leq 0},\Xcal^{\geq
0})$ (theorem \ref{glue}). 
\end{itemize}
In fact (see \cite{BeilinsonBernsteinDeligne82}, \cite{LiuVitoria12}), $(\Dcal^{\leq
0},\Dcal^{\geq 0})$ is glued with respect to $\Rcal$ if and only if  $j_!j^*(\Dcal^{\leq 0})
\subseteq \Dcal^{\leq 0}$ (or, equivalently, if and only if $j_*j^*(\Dcal^{\geq 0})
\subseteq \Dcal^{\geq 0}$)
in which case the \textbf{restrictions} to $\Ycal$ and $\Xcal$ satisfy
$$\Ycal^{\leq 0}=i^*(\Dcal^{\leq 0}),\quad \Ycal^{\geq 0}=i^!(\Dcal^{\geq 0}),$$
$$\Xcal^{\leq 0}=j^*(\Dcal^{\leq 0}),\quad \Xcal^{\geq 0}=j^*(\Dcal^{\geq 0}).$$

Let $\Acal_\Dcal$, $\Acal_\Ycal$ and $\Acal_\Xcal$ be the hearts of
these t-structures. In \cite[section 1.4]{BeilinsonBernsteinDeligne82}, it is shown that there is a \textbf{recollement of abelian categories} at the level of hearts
\begin{equation}\label{recollement ab}
\begin{xymatrix}{\Acal_\mathcal{Y}\ar[r]^{^p i_*}& \Acal_\mathcal{D}\ar@<3ex>[l]_{^pi^!}\ar@<-3ex>[l]_{^pi^*}\ar[r]^{^pj^*}&\Acal_\mathcal{\mathcal{X}}\ar@<3ex>_{^pj_*}[l]\ar@<-3ex>_{^pj_!}[l]}.
\end{xymatrix}
\end{equation}
We describe these functors explicitly. Let $\epsilon_\Dcal:\Acal_\Dcal\hookrightarrow\Dcal$ denote the
full embedding (similarly $\epsilon_\Ycal: \Acal_\Ycal\hookrightarrow \Ycal$ and $\epsilon_\Xcal:\Acal_\Xcal\hookrightarrow \Xcal$) and
let $H^i_\Dcal:\Dcal \ra \Acal_\Dcal$, $i\in\Zbb$, denote the  cohomological functors with respect to the fixed t-structure in $\Dcal$ 
(similarly $H^i_\Ycal$ and $H^i_\Xcal$ for the t-structures in $\Ycal$ and $\Xcal$). Then the
functors in the recollement (\ref{recollement ab}) are given by
$${^pi^*} = H^0_\Ycal\circ i^* \circ \epsilon_\Dcal,\ {^pi^!} = H^0_\Ycal\circ i^! \circ \epsilon_\Dcal, \
{^pi_*} = H^0_\Dcal\circ i_* \circ \epsilon_\Ycal,$$
$$ {^pj_!} = H^0_\Dcal\circ j_! \circ \epsilon_\Xcal,\ {^pj_*} = H^0_\Dcal\circ j_* \circ \epsilon_\Xcal, \
{^pj^*} = H^0_\Xcal\circ j^* \circ \epsilon_\Dcal.$$ See~\cite{FranjouPirashvili04} for more on recollements of abelian categories.

\begin{remark}\label{exact}
Since for our fixed t-structures, $i_*$ and $j^*$ are t-exact (see, for example, \cite{LiuVitoria12} for details), we have that ${^pi_*}=i_*\circ\epsilon_\Ycal$ and ${^pj^*}=j^*\circ\epsilon_\Dcal$. As a consequence, these two functors are exact.
\end{remark}
Moreover, for any object
$A$ in $\Acal_\Dcal$ there are exact sequences
\begin{equation}\label{ex seq 1}
0\ra {^pi_*}H^{-1}_\Ycal {^pi^*}A\ra {^pj_!}{^pj^*}A \ra A\ra
{^pi_*}{^pi^*}A \ra 0\,,
\end{equation}
\begin{equation}\label{ex seq 2}
0\ra {^pi_*}{^pi^!}A \ra A \ra
{^pj_*}{^pj^*}A \ra {^pi_*}H^{1}_\Ycal {^pi^*}A \ra 0\,.
\end{equation}
Torsion pairs   in $\Acal_\Dcal$ satisfying certain conditions can be
\textbf{restricted} to torsion pairs in $\Acal_\Ycal$ and $\Acal_\Xcal$.

\begin{lemma}\label{restriting torsion pair rhs} Let $(\Tcal,\Fcal)$ be a torsion pair in
$\Acal_\Dcal$. Then
\begin{enumerate}
\item $({^pi^*}(\Tcal),{^pi^!}(\Fcal))$ is a torsion
pair in $\Acal_\Ycal$;
\item the following are equivalent: 
\begin{itemize}
\item[(i)] $({^pj^*}(\Tcal),{^pj^*}(\Fcal))$ is a torsion pair in $\Acal_\Xcal$,
\item[(ii)] ${^pj_!}{^pj^*}(\Tcal) \subseteq \Tcal$,
\item[(iii)] ${^pj_*}{^pj^*}(\Fcal) \subseteq \Fcal$.
\end{itemize}
\end{enumerate}
\end{lemma}

\begin{proof} (1): 
We first check the orthogonality condition for $({^pi^*}(\Tcal),{^pi^!}(\Fcal))$.
We have that
$$\Hom_{\Acal_\Dcal}({^pi^*}(\Tcal),{^pi^!}(\Fcal)) = \Hom_{\Acal_\Dcal}({^pi_*}{^pi^*}(\Tcal), {^pi_*}{^pi^!}(\Fcal))$$
since $^pi_*$ is fully faithful. Consider the exact sequence (\ref{ex seq 1}) with $A$ in $\Tcal$ and apply to it 
the functor $\Hom_{\Acal_\Dcal}(-,B)$ for some $B$ in $\Fcal$. Since $\Hom_{\Acal_\Dcal}(A,B)=0$, we conclude that  
$\Hom({^pi_*}{^pi^*}(A),B)=0$, i.e., ${^pi_*}{^pi^*}(\Tcal)
\subseteq \Tcal$. Similarly we have ${^pi_*}{^pi^!}(\Fcal) \subseteq \Fcal$ and, thus, $\Hom_{\Acal_\Dcal}({^pi_*}{^pi^*}(\Tcal),{^pi_*}{^pi^!}(\Fcal))=0$, as wanted.

Secondly, we produce a suitable short exact sequence for any object in $\Acal_\Dcal$. By \cite[remark 1.4.17.1]{BeilinsonBernsteinDeligne82}, ${^pi_*}(\Acal_\Ycal)$ is a Serre subcategory of $\Acal_\Dcal$, i.e., for any
short exact sequence in $\Acal_\Dcal$,
$$0\ra K\ra M\ra L\ra 0,$$ $M$ lies in the ${^pi_*}(\Acal_\Ycal)$
if and only if so do $K$ and $L$.  Let $M$ lie in $\Acal_\Ycal$ and consider a short exact
sequence $$0\ra \widetilde{M_1}\ra {^pi_*}(M) \ra \widetilde{M_2} \ra
0$$ in $\Acal_\Dcal$ with $\widetilde{M_1}\in\Tcal$ and
$\widetilde{M_2}\in\Fcal$ (which exists since $(\Tcal,\Fcal)$ is a torsion pair in $\Dcal$). Since ${^pi_*}(\Acal_\Ycal)$ is a Serre
subcategory of $\Acal_\Dcal$, there exist $M_1$ and $M_2$ in $\Acal_\Ycal$
such that $\widetilde{M_i}={^pi_*}(M_i)$ for $i=1,2$. Now,
${^pi_*}$ is an exact full embedding (remark \ref{exact}) and, thus, we get a short exact sequence in
$\Acal_\Ycal$ $$0\ra M_1\ra M\ra M_2\ra 0,$$ 
with $M_1\cong{^pi^*}{^pi_*}(M_1)={^pi^*}(\widetilde{M_1})$ lying in
${^pi^*}(\Tcal)$ and $M_2\cong{^pi^!}{^pi_*}(M_2) =
{^pi^!}(\widetilde{M_2})$ lying in ${^pi^!}(\Fcal)$.

(2):  By remark \ref{exact}, the functor ${^pj^*}:\Acal_\Dcal\ra\Acal_\Xcal$ is exact. 
Hence, $({^pj^*}(\Tcal),{^pj^*}(\Fcal))$ is a torsion pair in
 $\Acal_\Xcal$ if and only these classes are orthogonal,
 i.e. $\Hom_{\Acal_\Xcal}({^pj^*}(\Tcal),{^pj^*}(\Fcal))=0$. Using the
 adjunctions $({^pj_!},{^pj^*})$ and $({^pj^*},{^pj_*})$, this holds if and only if
 $\Hom_{\Acal_\Dcal}({^pj_!}{^pj^*}(\Tcal),\Fcal)=0$ (equivalent to (ii)) and if and only
 if $\Hom_{\Acal_\Dcal}(\Tcal,{^pj_*}{^pj^*}(\Fcal))=0$ (equivalent to (iii)).
\end{proof}

When the conditions of part (2) of the lemma are fulfilled, we say that 
$(\Tcal,\Fcal)$ is \textbf{compatible} with the recollement (\ref{recollement ab}). We will see that this compatibility condition is precisely the requirement on a torsion pair so that the corresponding HRS-tilt of a glued t-structure is also obtained by glueing.

\begin{proposition} A torsion pair $(\Tcal,\Fcal)$ in $\Acal_\Dcal$ is compatible with
respect to (\ref{recollement ab}) if and only if the corresponding HRS-tilt of $(\Dcal^{\leq 0},\Dcal^{\geq 0})$ is a t-structure glued with respect to the recollement (\ref{recollement}).
\label{torsion pair iff HRS-tilt}
\end{proposition}

\begin{proof} Let $(\Tcal,\Fcal)$ be a torsion pair in
$\Acal_\Dcal$ and $(\widetilde{\Dcal}^{\leq
0},\widetilde{\Dcal}^{\geq 0})$ be the corresponding HRS-tilt, i.e., 
\begin{eqnarray*}
\widetilde{\Dcal}^{\leq 0} &=& \{X\in\Dcal\mid H^0(X)\in\Tcal,\  H^{i}(X)=0, \forall i>0\},\\
\widetilde{\Dcal}^{\geq 0} &= & \{X\in\Dcal\mid  H^{-1}(X)\in\Fcal,\ H^{i}(X)=0, \forall i<-1\},
\end{eqnarray*}
where $H^i, i\in\Zbb,$ denote the cohomological functors of the t-structure $(\Dcal^{\leq 0},\Dcal^{\geq 0})$.
By \cite{BeilinsonBernsteinDeligne82}, the HRS-tilt is glued if and only if
$j_!j^*(\widetilde{\Dcal}^{\leq 0}) \subseteq \widetilde{\Dcal}^{\leq
0}$. Denote by $\tau^{\leq -1}$ the truncation at $-1$ associated to 
$(\Dcal^{\leq 0},\Dcal^{\geq 0})$. For any $X\in\wt{\Dcal}^{\leq 0}$, there is a triangle in
$\Dcal$ of the form
$$\tau^{\leq -1}X\ra X\ra H^0 X \ra (\tau^{-1}X)[1],$$
since, by definition, $\wt{\Dcal}^{\leq
0}\subseteq \Dcal^{\leq 0}$.
Applying to it the functor $j_!j^*$, we get another triangle
$$j_!j^*(\tau^{\leq -1}X) \ra j_!j^*(X) \ra j_!j^*(H^0 X) \ra j_!j^*(\tau^{\leq -1}X)[1].$$
Since $(\Dcal^{\leq 0},\Dcal^{\geq 0})$ is glued,
$j_!j^*(\Dcal^{\leq -1})\subseteq\Dcal^{\leq -1}\subseteq\wt{\Dcal}^{\leq 0}$. Hence, $j_!j^*(\tau^{\leq -1}X)$ lies in $\wt{\Dcal}^{\leq 0}$ and, thus,
 $j_!j^*(X)$ lies in $\wt{\Dcal}^{\leq 0}$ if and only if
$j_!j^*(H^0X)$ lies in $\wt{\Dcal}^{\leq 0}$. Since $X$ is arbitrary, this means precisely that  $(\widetilde{\Dcal}^{\leq 0},\widetilde{\Dcal}^{\geq 0})$ can be restricted if and only if the torsion pair $(\Tcal,\Fcal)$ is compatible with (\ref{recollement ab}).
\end{proof}

The question that naturally follows is whether, under the compatibility condition on the torsion pair in $\Dcal$, 
the restrictions of the HRS-tilt of a glued t-strucutre in $\Dcal$
are precisely the HRS-tilts of the restricted t-structures on $\Ycal$ and
$\Xcal$ with respect to the restricted torsion pairs. 

\begin{theorem}\label{t:glueing-vs-tilt} Let $(\Tcal,\Fcal)$ be a torsion pair in $\Acal_\Dcal$, compatible with (\ref{recollement ab}) and with restrictions 
$(\Tcal_\Ycal,\Fcal_\Ycal)$ and $(\Tcal_\Xcal,\Fcal_\Xcal)$ in $\Acal_\Ycal$ and $\Acal_\Xcal$
respectively. Let $(\wt{\Ycal}^{\leq 0},\wt{\Ycal}^{\geq 0})$ and
$(\wt{\Xcal}^{\leq 0},\wt{\Xcal}^{\geq 0})$ be the corresponding
HRS-tilts of $(\Ycal^{\leq 0},\Ycal^{\geq 0})$ and $(\Xcal^{\leq
0},\Xcal^{\geq 0})$ respectively. Then the HRS-tilt
$(\wt{\Dcal}^{\leq 0},\wt{\Dcal}^{\geq 0})$ is obtained by glueing $(\wt{\Xcal}^{\leq 0},\wt{\Xcal}^{\geq 0})$ and $(\wt{\Ycal}^{\leq
0},\wt{\Ycal}^{\geq 0})$ with respect to the recollement (\ref{recollement}). \label{HRS-tilt at glued torsion pair}
\end{theorem}

\begin{proof} Let $H^i_\Dcal$, $H^i_\Xcal$ and $H^i_\Ycal$, $i\in\Zbb$, denote the cohomological functors associated with $(\Dcal^{\leq 0},\Dcal^{\geq 0})$, $(\Xcal^{\leq 0},\Xcal^{\geq 0})$ and $(\Ycal^{\leq 0},\Ycal^{\geq 0})$ respectively. We have to show that the following holds:
\begin{itemize}
\item[(i)] $j^*(\wt{\Dcal}^{\leq 0}) =\wt{\Xcal}^{\leq 0}$, $j^*(\wt{\Dcal}^{\geq 0}) =\wt{\Xcal}^{\geq
0}$;
\item[(ii)] $i^*(\wt{\Dcal}^{\leq 0}) = \wt{\Ycal}^{\leq 0}$, $i^!(\wt{\Dcal}^{\geq 0}) = \wt{\Ycal}^{\geq 0}$.
\end{itemize}
By definition, we have $$j^*(\wt{\Dcal}^{\leq 0}) =\{j^*X\mid X\in\Dcal,\ H^0_\Dcal(X)\in\Tcal,\ H^{i}_\Dcal(X)=0, \forall i>0\},$$  and, since $j^*$ is t-exact (\cite{BeilinsonBernsteinDeligne82}),
$H^i_\Xcal(j^*X)=j^*(H^i_\Dcal(X))$, thus showing that
$$j^*(\wt{\Dcal}^{\leq 0}) \subseteq \{j^*X\mid X\in\Dcal,\
H^0_\Xcal(j^*X)\in
j^*\Tcal,\ H^{i}_\Xcal(j^*X)=0,\forall i>0\}=\wt{\Xcal}^{\leq 0}.$$ Similarly, we get that
$$j^*(\wt{\Dcal}^{\geq 0}) \subseteq \{j^*X\mid X\in\Dcal,\
H^{-1}_\Xcal(j^*X)\in
j^*\Tcal,\ H^{i}_\Xcal(j^*X)=0,\forall i<-1\}=\wt{\Xcal}^{\geq 0}.$$ Since $(\Tcal,\Fcal)$ is compatible with the recollement, by proposition \ref{torsion pair
iff HRS-tilt}, the HRS-tilt $(\wt{\Dcal}^{\leq 0},\wt{\Dcal}^{\geq
0})$ can be restricted and, thus, $(j^*(\wt{\Dcal}^{\leq
0}),j^*(\wt{\Dcal}^{\geq 0}))$ is a t-structure in $\Xcal$. Since
$j^*(\wt{\Dcal}^{\leq 0})\subseteq \wt{\Xcal}^{\leq 0}$, we also have
that $j^*(\wt{\Dcal}^{\geq 0}) \supseteq \wt{\Xcal}^{\geq 0}$ and, thus,
$j^*(\wt{\Dcal}^{\geq 0})=\wt{\Xcal}^{\geq 0}$. This
proves (i).

To prove (ii), observe that $i^*$ is right t-exact (i.e., $i^*({\Dcal}^{\leq
0})\subseteq\Ycal^{\leq 0}$) and that
$$i^*(\wt{\Dcal}^{\leq 0})=\{i^*X\mid X\in\Dcal,\ H^0_\Dcal(X)\in\Tcal, \ H^{i}_\Dcal(X)=0,\forall i>0\}.$$

By \cite[proposition 1.3.17 (ii)]{BeilinsonBernsteinDeligne82}, for
$X$ in $\wt{\Dcal}^{\leq 0} \subseteq \Dcal^{\leq 0}$, we have
$H^0_\Ycal(i^*X)={^pi^*}H^0_\Dcal(X)$ which belongs to
${^pi^*}(\Tcal)$. Since $i^*(\wt{\Dcal}^{\leq 0})\subseteq
i^*({\Dcal}^{\leq 0}) \subseteq \Ycal^{\leq 0}$, we get that
$H^{i}_\Ycal(i^*\wt{\Dcal}^{\leq 0})=0$, for all $i>0$, and thus
$$i^*(\wt{\Dcal}^{\leq 0})\subseteq\{i^*X\mid X\in\Dcal,\ H^0_\Ycal(i^*X)\in {^pi^*}\Tcal,\ H^{i}_\Ycal(i^*X)=0,\forall i>0 \}=\wt{\Ycal}^{\leq 0}.$$
Analogously, we may conclude that 
$$i^*(\wt{\Dcal}^{\geq 0})\subseteq\{i^*X\mid X\in\Dcal,\ H^{-1}_\Ycal(i^*X)\in {^pi^*}\Fcal, H^{i}_\Ycal(i^*X)=0,\forall i<-1\}=\wt{\Ycal}^{\geq 0}.$$
The same argument as before shows that an equality holds, as wanted.
\end{proof}

The previous result showed that an HRS-tilt with respect to a compatible torsion pair is glued. We may now ask whether HRS-tilts on the sides of the recollement (\ref{recollement}) glue to an HRS-tilt in the middle. The following proposition answers this question positively.

\begin{proposition}\label{p:glueing-torsion-pair} Let $(\Tcal_\Ycal,\Fcal_\Ycal)$ and
$(\Tcal_\Xcal,\Fcal_\Xcal)$ be torsion pairs in $\Acal_\Ycal$ and
$\Acal_\Xcal$ respectively. Then the pair $(\Tcal,\Fcal)$ defined by
\begin{eqnarray*}
\Tcal &=& \{A\in\Acal_\Dcal\mid {^pi^*}A\in\Tcal_\Ycal,\
{^pj^*}A\in\Tcal_\Xcal\}\\
\Fcal &=& \{A\in\Acal_\Dcal\mid {^pi^!}A\in\Fcal_\Ycal,\
{^pj^*}A\in\Fcal_\Xcal\}.
\end{eqnarray*}
is a torsion pair in $\Acal_\Dcal$, compatible with (\ref{recollement ab}). Moreover, its restrictions are given by $(\Tcal_\Xcal,\Fcal_\Xcal)$ and $(\Tcal_\Ycal, \Fcal_\Ycal)$, and the HRS-tilt of $(\Dcal^{\leq
0},\Dcal^{\geq 0})$ with respect to $(\Tcal,\Fcal)$ is the glueing of the
HRS-tilts of $(\Ycal^{\leq 0},\Ycal^{\geq 0})$ and $(\Xcal^{\leq
0},\Xcal^{\geq 0})$ with respect to $(\Tcal_\Ycal,\Fcal_\Ycal)$ and
$(\Tcal_\Xcal,\Fcal_\Xcal)$ respectively.
\end{proposition}

\begin{proof}  We will use remark \ref{inter} to show that $(\Tcal,\Fcal)$ is a torsion pair, i.e., we will show that the pair $(\Dcal_{(\Tcal,\Fcal)}^{\leq 0},\Dcal_{(\Tcal,\Fcal)}^{\geq 0})$ formally defined as the HRS-tilt of $(\Dcal^{\leq 0},\Dcal^{\geq 0})$ with respect to $(\Tcal,\Fcal)$ is a t-structure (satisfying $\Dcal^{\leq -1}\subseteq\Dcal^{\leq 0}_{(\Tcal,\Fcal)}\subseteq \Dcal^{\leq 0}$). Denote by $(\wt{\Ycal}^{\leq 0},\wt{\Ycal}^{\geq 0})$ and
$(\wt{\Xcal}^{\leq 0},\wt{\Xcal}^{\geq 0})$ the HRS-tilts of
$(\Ycal^{\leq 0},\Ycal^{\geq 0})$
and $(\Xcal^{\leq 0},\Xcal^{\geq 0})$ at 
$(\Tcal_\Ycal,\Fcal_\Ycal)$ and $(\Tcal_\Xcal,\Fcal_\Xcal)$
respectively and let $(\wt{\Dcal}^{\leq 0},\wt{\Dcal}^{\geq
0})$ be the glueing of $(\wt{\Ycal}^{\leq 0},\wt{\Ycal}^{\geq
0})$ and $(\wt{\Xcal}^{\leq 0}, \wt{\Xcal}^{\geq 0})$  with respect to the fixed recollement, i.e.,
\begin{eqnarray*}
\wt{\Dcal}^{\leq 0} & = & \{X\in\Dcal \mid i^*(X)\in\wt{\Ycal}^{\leq 0},\,j^*(X)\in\wt{\Xcal}^{\leq 0}\} \\
\wt{\Dcal}^{\geq 0} & = & \{X\in\Dcal \mid i^!(X)\in\wt{\Ycal}^{\geq 0},\,j^*(X)\in\wt{\Xcal}^{\geq 0}\}. 
\end{eqnarray*}
We will show that $(\wt{\Dcal}^{\leq 0},\wt{\Dcal}^{\geq
0})=(\Dcal_{(\Tcal,\Fcal)}^{\leq 0},\Dcal_{(\Tcal,\Fcal)}^{\geq 0})$. As usual, $H^i_\Dcal$, $H^i_\Xcal$ and $H^i_\Ycal$, $i\in\Zbb$, denote the cohomological functors associated with $(\Dcal^{\leq 0},\Dcal^{\geq 0})$, $(\Xcal^{\leq 0},\Xcal^{\geq 0})$ and $(\Ycal^{\leq 0},\Ycal^{\geq 0})$, respectively.

First we show that $\Dcal_{(\Tcal,\Fcal)}^{\leq 0}$ is a subcategory of $\wt{\Dcal}^{\leq
0}$. Take $X$ in $\Dcal_{(\Tcal,\Fcal)}^{\leq 0}(\subseteq\Dcal^{\leq 0})$.
Since $i^*$ is right t-exact, $i^*(X)$ lies in $\Ycal^{\leq 0}$, i.e., $H^{i}_\Ycal(i^*X)=0$, for all $i>0$. By \cite[proposition 1.3.17(ii)]{BeilinsonBernsteinDeligne82},
$H^0_\Ycal(i^*X)={^p}i^*H^0_\Dcal(X)$ lies in ${^p}i^*(\Tcal)\subseteq\Tcal_\Ycal$, showing that $i^*(X)\in\wt{\Ycal}^{\leq 0}$. Since $j^*$ is t-exact, we have
$H^{i}_\Xcal(j^*X)=j^*H^{i}_\Dcal(X)=0$, for all $i>0$, and
$H^0_\Xcal(j^*X)=j^*H^0_\Dcal(X)$ lies in $j^*(\Tcal)\subseteq\Tcal_\Xcal$, i.e., $j^*X$ lies in $\wt{\Xcal}^{\leq 0}$ and, hence, $X$ lies in $\wt{\Dcal}^{\leq
0}$. 

Now we prove that $\wt{\Dcal}^{\leq 0}$ is a subcategory of $\Dcal_{(\Tcal,\Fcal)}^{\leq
0}$. Take $X$ in $\wt{\Dcal}^{\leq 0}$. 
By definition of HRS-tilt, we have that $\wt{\Ycal}^{\leq 0}\subset
\Ycal^{\leq 0}$ and $\wt{\Xcal}^{\leq 0}\subseteq \Xcal^{\leq 0}$.
This implies that $\wt{\Dcal}^{\leq 0}\subseteq \Dcal^{\leq 0}$ and, hence, $H^{i}_\Dcal(X)=0$, for all $i>0$.
By \cite{BeilinsonBernsteinDeligne82} (proposition 1.3.17(ii)), we have ${^pi^*}H^0_\Dcal(X) =
H^0_\Ycal(i^*X)$ which belongs to $\Tcal_\Ycal$. By the t-exactness of $j^*$,
${^pj^*}H^0_\Dcal(X)=H^0_\Xcal(j^*X)$ which, thus, belongs to
$\Tcal_\Xcal$. This means that, by definition of $\Tcal$, $H^0_\Dcal(X)$ belongs to
$\Tcal$. Thus $X$ belongs to $\wt{\Dcal}^{\leq 0}$ and $\wt{\Dcal}^{\leq 0}=\Dcal_{(\Tcal,\Fcal)}^{\leq 0}$, as wanted.

It remains to show that $(\Tcal,\Fcal)$ is compatible and that the restrictions are as expected. By lemma
\ref{restriting torsion pair rhs}, 
we need to show
${^pj_!}{^pj^*}\Tcal\subseteq\Tcal$. This is true since, for any
$A$ in $\Tcal$, ${^pi^*}{^pj_!}{^pj^*}A=0$ and
${^pj^*}{^pj_!}{^pj^*}A={^pj^*}A\subseteq\Tcal_\Xcal$. 
It follows that
$({^pi^*}(\Tcal),{^pi^!}(\Fcal))$ and
$({^pj^*}(\Tcal),{^pj^*}(\Fcal))$ are torsion pairs in $\Acal_\Ycal$ and
$\Acal_\Xcal$ respectively. On the other hand, by definition of
$(\Tcal,\Fcal)$, we have
$${^pi^*}(\Tcal)\subseteq\Tcal_\Ycal,\ \
{^pi^!}(\Fcal)\subseteq\Fcal_\Ycal,$$
$${^pj^*}(\Tcal)\subseteq\Tcal_\Xcal,\ \
{^pj^*}(\Fcal)\subseteq\Fcal_\Xcal,$$ and,  thus, the inclusions are actually equalities, as wanted. 
\end{proof}

We illustrate proposition \ref{p:glueing-torsion-pair} in a simple example.
\begin{example}
We adopt the notation in example~\ref{ex:hereditary} and consider the standard recollement there. 
Let $(\Dcal^{\leq 0},\Dcal^{\geq 0})$ be the standard t-structure on $\Dcal^b(R)$, which is depicted in the Auslander-Reiten quiver of $\Dcal^b(R)$ as
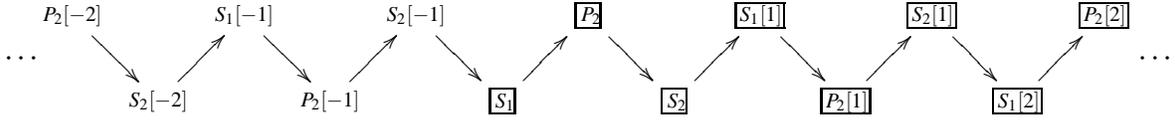
\[
\begin{xy} 0;<0.65pt,0pt>:<0pt,-0.65pt>::
(20,25) *+{\cdots} ="0", 
(50,0) *+{\makebox[21pt]{\tiny$P_2[-2]$}} ="1", 
(100,50) *+{\makebox[21pt]{\tiny $S_2[-2]$}} ="2",
(150,0) *+{\makebox[21pt]{\tiny$S_1[-1]$}}="3",
(200,50) *+{\makebox[21pt]{\tiny $P_2[-1]$}}="4",
(250,0) *+{\makebox[21pt]{\tiny$S_2[-1]$}}="5",
(300,50) *+{\framebox(9,9){\parbox{5pt}{\tiny$S_1$}}}="6",
(350,0) *+{\framebox(9,9){\parbox{5pt}{\tiny$P_2$}}}="7",
(400,50) *+{\framebox(9,9){\parbox{5pt}{\tiny$S_2$}}}="8",
(450,0) *+{\framebox(18,9){\parbox{15pt}{\tiny$S_1[1]$}}}="9",
(500,50) *+{\framebox(18,9){\parbox{15pt}{\tiny$P_2[1]$}}}="10",
(550,0) *+{\framebox(18,9){\parbox{15pt}{\tiny$S_2[1]$}}}="11",
(600,50) *+{\framebox(18,9){\parbox{15pt}{\tiny$S_1[2]$}} }="12",
(650,0) *+{ \framebox(18,9){\parbox{15pt}{\tiny$P_2[2]$}} }="13",
(680,25) *+{\cdots}="14",
 "1", {\ar"2"}, 
 "2", {\ar"3"},
 "3", {\ar"4"},
 "4", {\ar"5"},
 "5", {\ar"6"},
 "6", {\ar"7"},
 "7", {\ar"8"},
 "8", {\ar"9"},
 "9", {\ar"10"},
 "10", {\ar"11"},
 "11", {\ar"12"},
 "12", {\ar"13"},
\end{xy}
\]
where the objects in the boxes belong to the aisle $\Dcal^{\leq 0}$. 
This t-structure restricts to standard t-structures $(\Xcal^{\leq 0},\Xcal^{\geq 0})$ and $(\Ycal^{\leq 0},\Ycal^{\geq 0})$ on $\Dcal^b(eRe)$ and $\Dcal^b(R/ReR)$, respectively. Recall that both $eRe$ and $R/ReR$ are isomorphic to $\K$. Consequently, all bounded t-structures on $\Dcal^b(eRe)$ and $\Dcal^b(R/ReR)$ are shifts of the standard ones, and all mutation torsion pairs of $\mathrm{mod-}eRe$ and $\mathrm{mod-}R/ReR$ are trivial torsion pairs.

The HRS-tilt of $(\Xcal^{\leq 0},\Xcal^{\geq 0})$ with respect to the torsion pair $(0,\mathrm{mod-}eRe)$ is $(\Xcal^{\leq 0},\Xcal^{\geq 0})$, and the HRS-tilt of $(\Ycal^{\leq 0},\Ycal^{\geq 0})$ with respect to the torsion pair $(\mathrm{mod-}R/ReR,0)$  is $(\Ycal^{\leq 0}[1],\Ycal^{\geq 0}[1])$. These two new t-structures are glued, via the standard recollement,   to a t-structure $(\Dcal'^{\leq 0},\Dcal'^{\geq 0})$ on $\Dcal^b(R)$, which is depicted as
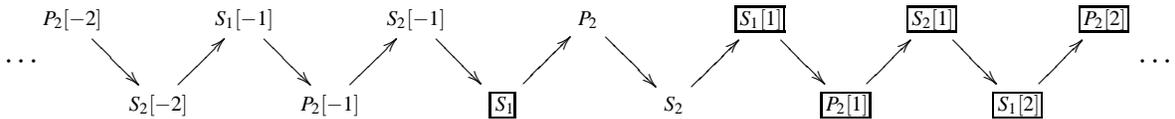
\[
\begin{xy} 0;<0.65pt,0pt>:<0pt,-0.65pt>::
(20,25) *+{\cdots} ="0", 
(50,0) *+{\makebox[21pt]{\tiny$P_2[-2]$}} ="1", 
(100,50) *+{\makebox[21pt]{\tiny $S_2[-2]$}} ="2",
(150,0) *+{\makebox[21pt]{\tiny$S_1[-1]$}}="3",
(200,50) *+{\makebox[21pt]{\tiny $P_2[-1]$}}="4",
(250,0) *+{\makebox[21pt]{\tiny$S_2[-1]$}}="5",
(300,50) *+{\framebox(9,9){\parbox{5pt}{\tiny$S_1$}}}="6",
(350,0) *+{\makebox[9pt]{\tiny$P_2$}}="7",
(400,50) *+{\makebox[9pt]{\tiny$S_2$}}="8",
(450,0) *+{\framebox(18,9){\parbox{15pt}{\tiny$S_1[1]$}}}="9",
(500,50) *+{\framebox(18,9){\parbox{15pt}{\tiny$P_2[1]$}}}="10",
(550,0) *+{\framebox(18,9){\parbox{15pt}{\tiny$S_2[1]$}}}="11",
(600,50) *+{\framebox(18,9){\parbox{15pt}{\tiny$S_1[2]$}} }="12",
(650,0) *+{ \framebox(18,9){\parbox{15pt}{\tiny$P_2[2]$}} }="13",
(680,25) *+{\cdots}="14",
 "1", {\ar"2"}, 
 "2", {\ar"3"},
 "3", {\ar"4"},
 "4", {\ar"5"},
 "5", {\ar"6"},
 "6", {\ar"7"},
 "7", {\ar"8"},
 "8", {\ar"9"},
 "9", {\ar"10"},
 "10", {\ar"11"},
 "11", {\ar"12"},
 "12", {\ar"13"},
\end{xy}
\]
where the objects in the boxes belong to the aisle $\Dcal'^{\leq 0}$. The heart of $(\Dcal'^{\leq 0},\Dcal'^{\geq 0})$ is $add(S_1\oplus S_2[1])$. It is easy to check that $(\Dcal'^{\leq 0},\Dcal'^{\geq 0})$ is the HRS-tilt of $(\Dcal^{\leq 0},\Dcal^{\geq 0})$ at the torsion pair $(add(S_1),add(S_2))$, which is glued from the torsion pairs $(0,\mathrm{mod-}eRe)$  and $(\mathrm{mod-}R/ReR,0)$.
\end{example}

By subsection \ref{HRS-tilts and silting mutation}, we know that irreducible silting mutations correspond to HRS-tilts (see theorem \ref{mutations commute}) with respect to mutation torsion pairs in the heart (i.e., torsion pairs in which either the torsion class or the torsion-free class are given by the additive closure of a simple object). We analyse when such torsion pairs are compatible with (\ref{recollement ab}).

\begin{remark}\label{simples in the heart}
Recall that there is a functor $j_{!*}:\Acal_\Xcal\rightarrow \Acal_\Dcal$, called the intermediate image functor, such that every simple object in $\Acal_\Dcal$ is either of the form $i_*S$ for some simple object in $\Acal_\Ycal$ or of the form $j_{!*}S$ for some simple object in $\Acal_\Xcal$. For details, check \cite[section 1.4]{BeilinsonBernsteinDeligne82}.
\end{remark} 

\begin{proposition} Suppose that $\Acal_\Dcal$ is a  length category. 
Let $(\Tcal,\Fcal)$ be a mutation torsion pair in $\Acal_\Dcal$. Then $(\Tcal, \Fcal)$ is compatible with the recollement (\ref{recollement ab}) if and only if one of the following holds.
\begin{enumerate}
\item $\Tcal=$ add$(i_*S_Y)$ for some simple object $S_Y$ in $\Acal_\Ycal$;
\item $\Fcal=$ add$(i_*S_Y)$ for some simple object $S_Y$ in $\Acal_\Ycal$;
\item there is a simple object $S_X$ in $\Acal_\Xcal$ such that ${^pj_!}S_X$ is simple and $\Tcal=$ add$({^pj_!}S_X)$;
\item there is a simple object $S_X$ in $\Acal_\Xcal$ such that ${^pj_*}S_X$ is simple and $\Fcal=$ add$({^pj_*}S_X)$.
\end{enumerate}
Moreover, in each case, the restrictions are a trivial torsion pair on one side and a mutation torsion pair on the other.
\end{proposition}
\begin{proof}
This can be showed using lemma \ref{restriting torsion pair rhs} on a case by case analysis of the possible mutation torsion pairs in $\Acal_\Dcal$. Using remark \ref{simples in the heart} we get the following four cases, where $S_Y$ is a simple object in $\Acal_\Ycal$ and $S_X$ is a simple object in $\Acal_\Xcal$.
\begin{enumerate}
\item Suppose that $\Tcal=$ add$(i_*S_Y)$. Since $j^*i_*=0$, we have that ${^pj_!}j^*\Tcal=0$ and, thus, the torsion pair $(\Tcal,\Fcal)$ is compatible with the recollement. The restriction to $\Xcal$ is $(0, \Xcal)$ and the restriction to $\Ycal$ is the mutation torsion pair with torsion class $\Tcal_\Ycal=$ add$(S_Y)$.
\item Suppose that $\Fcal=$ add$(i_*S_Y)$. Since $j^*i_*=0$, we have that ${^pj_*}j^*\Tcal=0$ and, thus, the torsion pair $(\Tcal,\Fcal)$ is compatible with the recollement. The restriction to $\Xcal$ is  $(\Xcal, 0)$ and the restriction to $\Ycal$ is the mutation torsion pair with torsion-free class $\Fcal_\Ycal=$ add$(S_Y)$.
\item Suppose that $\Tcal=$ add$(j_{!*}S_X)$. By lemma \ref{restriting torsion pair rhs} the torsion pair $(\Tcal,\Fcal)$ is compatible with the recollement if and only if $^pj_! j^*(j_{!*}S_X)$ lies in $add(j_{!*}S_X)$. Since $j^*j_{!*}$ is naturally equivalent to the identity functor, this amounts to $^pj_! S_X$ lying in $add(j_{!*}S_X)$, which is equivalent to $^pj_! S_X\cong j_{!*}S_X$ because $j^* {}^pj_!$ is also naturally equivalent to the identity functor. The last condition holds if and only if $^pj_! S_X$ is simple. This follows from the fact that there is always a natural epimorphism ${^pj_!}S_X\ra j_{!*}S_X$. The restriction to $\Ycal$ is the pair $(0, \Ycal)$ and the restriction to $\Xcal$ is the mutation torsion pair given by the torsion class $\Tcal_\Xcal=$ add$(S_X)$.
\item Suppose that $\Fcal=$ add$(j_{!*}S_X)$. By lemma \ref{restriting torsion pair rhs} the torsion pair $(\Tcal,\Fcal)$ is compatible with the recollement if and only if $^pj_* j^*(j_{!*}S_X)$ lies in $add(j_{!*}S_X)$. Since $j^*j_{!*}$ is naturally equivalent to the identity functor, this amounts to $^pj_* S_X$ lying in $add(j_{!*}S_X)$, which is equivalent to $^pj_* S_X\cong j_{!*}S_X$ because $j^* {}^pj_*$ is also naturally equivalent to the identity functor. The last condition holds if and only if $^pj_* S_X$ is simple. This follows from the fact that there is always a natural monomorphism $j_{!*}S_X\ra {^pj_*}S_X$. The restriction to $\Ycal$ is the pair $(\Ycal,0)$ and the restriction to $\Xcal$ is the mutation torsion pair given by the torsion-free class $\Fcal_\Xcal=$ add$(S_X)$.
\end{enumerate}
The proof is complete.
\end{proof}

Suppose now that $\Dcal$ has a Serre functor. We finish this section with an observation on the compatibility of irreducible silting mutation with glueing (via the recollement $\Rcal_U$, since our focus is on t-structures rather than co-t-structures). 
Note that the items (3) and (4) of the proposition above rarely occur, because, in general,  neither $^pj_* S_X$ nor $^pj_!S_Y$ are simple in $\Acal_\Dcal$. Translating items (1) and (2) of the proposition in terms of irreducible silting mutation, we get the following corollary, also appearing in \cite{AiharaIyama12} (lemma 2.40).

\begin{corollary}
Let $Y$ and $X$ be the silting objects associated to the fixed bounded t-structures in $\Ycal$ and $\Xcal$, respectively. Suppose $Y'$ is an indecomposable direct summand of $Y$. Let $Z$ be the glued silting object with respect to the recollement $\Rcal_U$ (i.e., compatible with the glueing of t-structures via $\Rcal$) of $Y$ and $X$. Then there is an indecomposable direct summand $Z'$ of $K_Y$ such that the glued silting (with respect to $\Rcal_U$) of $\mu^+_{Y'}Y$ (respectively, $\mu^-_{Y'}Y$) and $X$ is precisely $\mu^+_{Z'}Z$ (respectively, $\mu^-_{Z'}Z$).
\end{corollary}

\def\cprime{$'$}
\providecommand{\bysame}{\leavevmode\hbox to3em{\hrulefill}\thinspace}
\providecommand{\MR}{\relax\ifhmode\unskip\space\fi MR }
\providecommand{\MRhref}[2]{
  \href{http://www.ams.org/mathscinet-getitem?mr=#1}{#2}
}
\providecommand{\href}[2]{#2}

\end{document}